\documentclass[11pt]{article}
\usepackage[a4paper,margin=2.5cm]{geometry}
\usepackage[comma, round]{natbib}
\usepackage{float}
\usepackage{amssymb}
\usepackage{amsthm}
\usepackage{amsmath}
\usepackage{titletoc}
\usepackage{mathrsfs}
\usepackage[colorlinks, allcolors=blue]{hyperref}
\usepackage{amsfonts}
\usepackage{graphicx}
\usepackage[linesnumbered, ruled, vlined]{algorithm2e}
\usepackage{esint}
\usepackage{bbm}
\usepackage{bm}
\usepackage{mathtools}
\usepackage[shortlabels, inline]{enumitem}
\usepackage{authblk}
\usepackage{todonotes}
\usepackage{color}

\numberwithin{equation}{section}
\numberwithin{figure}{section}

\theoremstyle{plain}
\newtheorem{thm}{Theorem}[section]

\newtheorem{lem}[thm]{Lemma}
\newtheorem{prop}[thm]{Proposition}

\theoremstyle{definition}
\newtheorem{defn}[thm]{Definition}
\newtheorem{asmp}[thm]{Assumption}
\newtheorem{rmk}[thm]{Remark}
\newtheorem{exam}[thm]{Example}

\newcommand{\cB}{\mathcal{B}}

\newcommand{\cF}{\mathcal{F}}
\newcommand{\cG}{\mathcal{G}}

\newcommand{\cP}{\mathcal{P}}
\newcommand{\cQ}{\mathcal{Q}}

\newcommand{\cT}{\mathcal{T}}

\newcommand{\cX}{\mathcal{X}}
\newcommand{\cY}{\mathcal{Y}}

\newcommand{\bF}{\mathbf{F}}

\newcommand{\bbD}{\mathbb{D}} 
\newcommand{\bbE}{\mathbb{E}}

\newcommand{\bbR}{\mathbb{R}}

\newcommand{\scrH}{\mathscr{H}}

\newcommand{\scrM}{\mathscr{M}}

\newcommand{\scrP}{\mathscr{P}}

\newcommand{\scrS}{\mathscr{S}}

\newcommand{\1}{\mathbbm{1}}

\newcommand{\la}{\langle}
\newcommand{\ra}{\rangle}

\newcommand{\md}{\mathop{}\mathopen\mathrm{d}}

\newcommand{\E}{E}

\newcommand{\Id}{\operatorname{Id}}

\newcommand{\M}{\mathrm{Mart}}

\newcommand{\proj}{\operatorname{pj}}

\newcommand{\Law}{\operatorname{Law}}

\newcommand{\mD}{\mathbb{D}}
\newcommand{\MD}{\mathbf{D}}

\renewcommand{\c}{\mathfrak{c}}
\newcommand{\bc}{\mathfrak{bc}}

\newcommand{\op}{\prescript{\mathrm{o}}{}}
\newcommand{\pp}{\prescript{\mathrm{p}}{}}

\newcommand{\lt}{[\![}
\newcommand{\rt}{]\!]}
\newcommand{\os}{\mathrm{OS}}

\def\lb{\mathopen{}\mathclose\bgroup\left}
\def\rb{\aftergroup\egroup\right}

\title{Sensitivity of causal distributionally robust optimization}
\author{Yifan Jiang\thanks{Email: {\tt yifan.jiang@maths.ox.ac.uk}.} }
\author{Jan Ob\l\'oj\thanks{Email: {\tt jan.obloj@maths.ox.ac.uk}.}}
\affil{Mathematical Institute, University of Oxford}
\date{}

\begin{document}

\maketitle

\begin{abstract}
    We study the causal distributionally robust optimization (DRO) in both discrete- and continuous- time settings.
    The framework captures model uncertainty, with potential models penalized in function of their adapted Wasserstein distance to a given reference model. Strength of the penalty is controlled using a real-valued parameter which, in the special case of an indicator penalty,
    is simply the radius of the uncertainty ball.
    Our main results derive the first-order sensitivity of the value of causal DRO with respect to the penalization parameter, i.e., we compute the sensitivity to model uncertainty.
    Moreover, we investigate the case where a martingale constraint is imposed on the underlying model, as is the case for pricing measures in mathematical finance.
    We introduce different scaling regimes, which allow us to obtain the continuous-time sensitivities as nontrivial limits of their discrete-time counterparts. We illustrate our results with examples. The sensitivities are naturally expressed using optional projections of Malliavin derivatives. To establish our results we obtain several novel results which are of independent interest. In particular, we introduce pathwise Malliavin derivatives and show these extend the classical notion. We also establish a novel stochastic Fubini theorem.
    \medskip

    \noindent{\em Keywords}: distributionally robust optimization, sensitivity analysis, causal optimal transport, pathwise Malliavin calculus, forward integral.
    \medskip
\end{abstract}

\section{Introduction}
Stochastic modelling in a dynamic setting is ubiquitous and is a staple of applied mathematics in many domains.
Traditionally, models were derived from theoretical considerations, combined with calibration to data, and often had tractable analytic representations.
More recently, models are often obtained through a data-driven approach, where they are equivalent to, or derived from, the empirical distribution of the sample paths.
In both scenarios, however, the model can easily be misspecified, leading to a discrepancy between the postulated model and the ground-truth distribution.
Therefore, it is of fundamental importance to understand the impact of model uncertainty on the performance of the model.

A distributionally robust approach is particularly appealing in this context.
It is formulated as a minimax problem where the agent seeks the optimal decision under the worst-case distribution from a neighbourhood around the reference model.
In this paper, we focus on the inner maximization  and study  the distributionally robust optimization (DRO) in a dynamic context
\begin{equation}
    \label{eqn-cdro}
    \sup_{\nu\in B_{\delta}(\mu)}\E_{\nu}[f(X)],
\end{equation}
where \(X=(X_{t})_{t\in I}\) is a stochastic process with paths in \(\cX\), \(f:\cX\to\bbR\) a path-dependent function(nal), \(\mu\in\scrP(\cX)\) the reference model, and  \(B_{\delta}(\mu)\) a suitable \(\delta\)-neighbourhood of \(\mu\).
The neighbourhood, i.e., the ambiguity set, is pivotal in the design of DRO and is often tailored to satisfy the statistical constraints of specific applications.
We refer to \citet{rahimianDistributionallyRobustOptimization2019} for a comprehensive exposition.
Among various choices of ambiguity sets, the optimal transport based discrepancy has drawn a great attention in operation research \citep{mohajerin2018data,blanchetQuantifyingDistributionalModel2019,gao2023distributionally,gao2023finite}, mathematical finance \citep{obloj2021distributionally,nendel22parametric,wu2023robust} and machine learning \citep{blanchet_kang_murthy_2019,bai2024wasserstein,NietertWDRO}.

With a suitable cost, optimal transport induces the so-called Wasserstein distance which metrizes the weak topology on probability space.
However, it is well known that the weak topology is too coarse to distinguish two processes with similar laws but different filtrations.
For example, the value of optimal stopping problems is not a continuous function of the law of the process under the weak topology, see \citet{bartlSensitivityMultiperiodOptimization2022}.
This led to many refined topologies and related notions, including \citet{vershik1970decreasing}'s iterated Kantorovich distance,  \citet{aldousWeakConvergenceGeneral1981}'s extended weak topology, \citet{hoover1984adapted}'s adapted equivalence, \citet{hellwigSequentialDecisionsUncertainty1996}'s information topology, \citet{pflug2012distance}'s nested distance and the more recent (bi)causal/adapted Wasserstein distance \citep{pammer24note,bartlWassersteinSpaceStochastic2021} - with all the above topologically equivalent as shown in
\citet{backhoff-veraguasAllAdaptedTopologies2020}. Motivated by their observation, we consider the discrepancy given by a \emph{causal} optimal transport problem
\begin{equation*}
    d_{\c}(\mu,\nu):=\inf_{\pi\in\Pi_{\c}(\mu,\nu)}\E_{\pi}[c(X,Y)],
\end{equation*}
which metrizes the adapted weak topology, and hence fully characterizes the temporal and spatial information of stochastic processes.
Here, the infimum is taken over a set of \emph{causal} couplings which was first systematically studied in \citet{lassalleCausalTransportPlans2018} and in more recent works \citet{acciaioCausalOptimalTransport2020,bartlWassersteinSpaceStochastic2021,beiglbock22Denseness,beiglbock23Representing,bartl25Wasserstein}.
Roughly speaking, the \emph{causality} constraint excludes the transport plans which can peek into the future of the model dynamic.

We fix \(p>1\), take \(B_{\delta}(\mu)=\{\nu\in\scrP(\cX):d_{\c}(\mu,\nu)\leq \delta^{p}\}\), and extend \eqref{eqn-cdro} to a general penalized causal DRO problem
\begin{equation}
    \label{eqn-cdro-pen}
    V(\delta)=\sup_{\nu\in \scrP(\cX)}\lb\{\E_{\nu}[f(X)]-L_{\delta}(d_{\c}(\mu,\nu)^{1/p})\rb\},
\end{equation}
where \(L_{\delta}(\cdot)=\delta L(\cdot/\delta)\) is a parametrized penalization.
Notice that \eqref{eqn-cdro} is of this form by taking \(L=+\infty\1_{(1,\infty)}\).
Moreover, we are interested in the case where an extra martingale constraint is imposed on the underlying model.
We assume the reference model \(\mu\) is a martingale measure and write
\begin{equation}
    \label{eqn-cdro-mart}
    V_{\M}(\delta)=\sup_{\nu\in \scrM(\cX)}\lb\{\E_{\nu}[f(X)]-L_{\delta}(d_{\bc}(\mu,\nu)^{1/p})\rb\},
\end{equation}
where \(\scrM(\cX)\) denotes the set of martingale measures.
For technical reasons, we replace the causal transport cost in the penalization with the \emph{bicausal} transport cost \(d_{\bc}(\mu,\nu)\).
Such a problem is of central interest in the context of robust finance, which gives the deviation of the price of financial derivatives in a misspecified non-arbitrage market.

To study \eqref{eqn-cdro-pen} and \eqref{eqn-cdro-mart}, there are two lines of research: duality and sensitivity.
The duality approach gives an alternative representation of the exact DRO value.
We refer to the works of \citet{blanchetQuantifyingDistributionalModel2019,gao2023distributionally,zhangSimpleGeneralDuality2022,shafiee2024newperspectives} for the static case and \citet{jiang24Duality} for the dynamic case.
In this paper, we follow the sensitivity approach  which aims to provide an explicit first order expansion of the DRO value with respect to the penalization parameter \(\delta\).
In the static setting, it was first derived in \citet{bartlSensitivityAnalysisWasserstein2021} with \(L=+\infty\1_{(1,\infty)}\) and extended to a general penalization in \citet{nendel22parametric}.
Specific dynamic Markovian settings have been recently tackled in \cite{fuhrmann2023wasserstein,langner2024markovnashequilibriameanfieldgames,NeufeldQLearning,mirmominov2024dynamicprogrammingprinciplemultiperiod}.
Partial discrete-time sensitivity results were derived independently from this paper in \citet{bartlSensitivityMultiperiodOptimization2022} for the non-constraint case, and in  \citet{sauldubois24First}  for the martingale-constraint case, see the discussion below. In \citet{bartl2023sensitivity}, the authors worked in a continuous time setting with models represented by It\^o diffusions and considered an ambiguity set defined via drift and volatility coefficients, see Section \ref{sec:further} for discussion.

We set up the discrete-time framework as follows.
Let \(I=\{0,1,\dots,N\}\) and \(\cX=\{0\}\times (\bbR^{d})^{N}\) equipped with its natural filtration \((\cF_{n})_{n\in I}\).
We fix \(p>1\) and \(q=p/(p-1)\).
By \(|\cdot|\) we denote the \(l_{p}\)-norm on \(\bbR^{d}\) and by \(|\cdot|_{*}\) we denote the \(l_{q}\)-norm on \(\bbR^{d}\).
Let \(\Delta:\cX\to\cX\) be the increment map given by \((0,x_{1},\dots,x_{N})\mapsto (0,x_{1},\dots,x_{N}-x_{N-1})\).
We take \(d_{\c}\) as the causal transport cost induced by the cost
\begin{equation*}
    c_{N}(x,y)=\sum_{n=1}^{N}|\Delta x_{n}- \Delta y_{n}|^{p}.
\end{equation*}
In Theorem \ref{thm-sens} we show that under mild regularity conditions the sensitivity of \eqref{eqn-cdro-pen} is given by
\begin{equation}
    \label{eqn-sens}
    \Upsilon:=\lim_{\delta\to 0}\frac{1}{\delta}(V(\delta)-V(0))= L^{*}\lb(\E_{\mu}\lb[\sum_{n=1}^{N}\lb|\E_{\mu}[\mD_{n}f (X)|\cF_{n}]\rb|_{*}^{q}\rb]^{1/q}\rb),
\end{equation}
where \(L^{*}\) is the convex conjugate of \(L\), and  \(\mD=(\mD_{1},\dots,\mD_{N})\) is the pullback of \(\nabla=(\partial_{1},\dots,\partial_{N})\) under \(\Delta\).
For the martingale constraint problem \eqref{eqn-cdro-mart}, in the case \(p=2\), it is shown in Theorem~\ref{thm-sens-mart} that
\begin{align}
    \label{eqn-sens-mart}
    \Upsilon_{\M}: & =\lim_{\delta\to 0}\frac{1}{\delta}(V_{\M}(\delta)-V_{\M}(0))                                                                  \nonumber \\
                   & = L^{*}\lb(\E_{\mu}\lb[\sum_{n=1}^{N}\lb|\E_{\mu}[\mD_{n}f (X)|\cF_{n}]-\E_{\mu}[\mD_{n}f (X)|\cF_{n-1}]\rb|_{*}^{2}\rb]^{1/2}\rb).
\end{align}
Our discrete-time results \eqref{eqn-sens} and \eqref{eqn-sens-mart} thus offer a multistep extension to the results in \citet{bartlSensitivityAnalysisWasserstein2021}.
We note that a variant of the unconstrained case \eqref{eqn-sens}  was established independently in \citet{bartlSensitivityMultiperiodOptimization2022} with \(L=+\infty\1_{(1,\infty)}\) and \(c(x,y)=\sum_{n=1}^{N}|x_{n}-y_{n}|^{p}\), see Remark \ref{rmk:linkwithDaniJohannes} for a detailed discussion.

With the discrete-time results in hand, we investigate the corresponding continuous-time limits in different scaling regimes.
We focus on stochastic processes with paths in \(\cX=C_{0}([0,T];\bbR^{d})\), i.e., the space of continuous paths starting at \(0\).
For any path \(\omega\in\cX\), we denote by \(\omega^{N}=(0,\omega_{T/N},\dots,\omega_{T})\) its discretization.
We  first  introduce the \emph{hyperbolic} scaling
\begin{equation*}
    c(\omega,\eta):=\lim_{N\to\infty}N^{p-1}c_{N}(\omega^{N},\eta^{N})= \|\omega-\eta\|^{p}_{W^{1,p}_{0}},
\end{equation*}
where the Sobolev norm is given by \(\|\omega\|_{W^{1,p}_{0}}=\|\dot{\omega}\|_{L^{p}}\).
Under such scaling, we show in Theorem~\ref{thm-cont} that the sensitivity of \eqref{eqn-cdro-pen} is given by
\begin{equation}
    \label{eqn-sens-cont}
    \Upsilon=\lim_{\delta\to 0}\frac{1}{\delta}(V(\delta)-V(0))=L^{*}\lb(\E_{\mu}\lb[\int_{0}^{T}|\E_{\mu}[\mD_{t}f(X)|\cF_{t}]|_{*}^{q}\md t\rb]^{1/q}\rb),
\end{equation}
where \(\mD_{t}\) is the  \emph{pathwise} Malliavin derivative (intuitively) given by
\begin{equation*}
    \la \mD_{t}f(\omega),e\ra :=\lim_{\varepsilon\to 0}\frac{f(\omega+\varepsilon e\1_{[0,t]})-f(\omega)}{\varepsilon}.
\end{equation*}
Indeed, \eqref{eqn-sens-cont} is a natural continuous-time limit of \eqref{eqn-sens} if we view \(\mD_{t}\) as a continuous-time extension to \(\mD_{n}\).
We defer the formal definition of \(\mD\) and the relation to the classical Malliavin derivative to Section \ref{sec-pre}. However, a similar limiting argument under the martingale constraint would feature the difference of the optional and predictable projections, which is known to be a thin process
and would thus give \(\Upsilon_{\M}=0\). Alternatively, this is also clear from Doob's decomposition: processes at a finite bi-causal distance under \( c(\omega,\eta)=\|\omega-\eta\|^{p}_{W^{1,p}_{0}}\) will only differ by their finite variation part. In fact, there is no other martingale measure in any Wasserstein (bi)causal neighbourhood of a given martingale measure \(\mu\).

Put differently, the \emph{hyperbolic} scaling is not critical for the martingale constraint problem. To allow ambiguity in the volatility, we have to zoom out and introduce the \emph{parabolic} scaling  (heuristically) given  by
\begin{equation*}
    c(\omega,\eta)=\lim_{N\to\infty}c_{N}(\omega^{N},\eta^{N})= [\omega-\eta]_{T},
\end{equation*}
where \([\cdot]_{T}\) denotes the scalar quadratic variation at terminal time \(T\).
We take  \(\mu\) as the Wiener measure, denote the classical Malliavin derivative by \(\MD\) and  focus on the objective of the form of
\begin{equation*}
    f(X)=U\lb(H\rb),\quad \textrm{for } H=\int_{0}^{T} \la h(t,X_{\cdot\wedge t}), \md X_{t}\ra,
\end{equation*}
for reasons explored in Remark \ref{rmk:typeofpayoffmg}. In Theorem \ref{thm-cont-mart},  we derive the sensitivity
\[\Upsilon_{\M}=L^{*}\lb(\E_{\mu}\lb[\int_{0}^{T}\|\varphi_{t}\|_{\bF}^{2}\md t\rb]^{1/2}\rb),\]
where
\begin{equation*}
    \varphi_{t}= \prescript{\mathrm{p}}{}{\lb\{\MD_{t+}\MD_{t}U(H)-  U'(H) \MD_{t}^{\intercal}h(t,X_{\cdot\wedge t}) \rb\}},
\end{equation*}
and \({}^{\intercal}\) denotes the linear transpose.

The terms \emph{hyperbolic} and \emph{parabolic} scaling are borrowed from the literature on hydrodynamic limits of interacting particle systems, see \cite{kipnis99Scaling}.
Roughly speaking, if the microparticle system has a non-zero mean, its distribution profile converges to a macroscopic continuity equation under hyperbolic scaling, characterizing a drift process.
Conversely, if the particle system is centreed, its distribution profile converges to a macroscopic Fokker-Planck equation  under parabolic scaling, characterizing a diffusion process.
In our context, hyperbolic scaling corresponds to the model's drift uncertainty, while parabolic scaling corresponds to its volatility uncertainty.
In particular, the adverse distribution is approximately a perturbation of the reference model by a drift process and a diffusion process, respectively.

For the continuous-time results, both scaling regimes are novel since we do not require an inner product structure on \(C_{0}([0,T];\bbR^{d})\).
Under hyperbolic scaling, essentially we derive a \emph{pathwise} first order expansion along any path \(h\) in the Sobolev space \(W_{0}^{1,p}([0,T];\bbR^{d})\).
We introduce a \emph{pathwise} Malliavin derivative such that the directional derivative of \(f\) along \(\eta\) can be represented as \(\la\mD f,\dot{\eta}\ra\).
For parabolic scaling, we study the property of the forward integral against a martingale \(\int_{0}^{T} \la \Phi_{t}, \md^{-} M_{t}\ra\).
A new stochastic Fubini theorem is established in Theorem~\ref{thm-fubini} which states that we can swap the order of the forward integral and the It\^o integral with an extra correction term.
The sensitivity  \(\Upsilon_{\M}\) then follows from a first order expansion in Lemma~\ref{lem-wiener}.

To the best of our knowledge, this is the first work that studies the continuous-time DRO sensitivity as a limit of its discrete-time counterpart.
The closest setting in continuous-time is the recent work \citet{bartl2023sensitivity} in which the authors consider $L^p$-balls around the drift and the volatility of the reference model under  a strong formulation. We discuss the link between our works in detail in Section \ref{sec:further} below. \cite{herrmann2017model,herrmann2017hedging} also considered related sensitivities albeit in a more specific setup with the penalty in \eqref{eqn-cdro-pen} depending on the payoff in a way which makes the sensitivity universal and independent of agent's risk aversion.

Naturally, the issue of model uncertainty and its impact on agents' course of action is a topic with a long history, including in decision theory and mathematical finance. Any lecture in mathematical finance covering derivatives' pricing, will typically also cover the `Greeks', the sensitivities of the prices to key model parameters, used throughout the financial industry.
Parametric, or similarly specific, sensitivities of optimal investment problems have been considered in a number of works, see \cite{larsen2007stability,mostovyi2019sensitivity} and the references therein. On a more abstract level, robust utility maximization corresponds to our problem \eqref{eqn-cdro-pen}, with an additional external optimization over controls and possibly with different penalties. It benefits from an axiomatic decision-theoretic justification, going back to \cite{gilboa89,maccheroni2006ambiguity}. Its analytic properties have been studied in depth, see, for example, \cite{tevzadze2013robust,neufeld2018robust} for existence of optimal strategies, \cite{schied,gordan,kallblad2018dynamically} for formulation of a dynamic programming principle and \cite{matoussi2015robust} for relation to 2BSDEs. However, except special cases, these works rarely allow for any explicit computations.

The rest of the paper is organized as follows.
In Section \ref{sec-pre}, we introduce the necessary notations and concepts.
In Section \ref{sec-tools}, we introduce a notion of \emph{pathwise} Malliavin derivative and discuss some properties of forward integral when the integrator is a martingale.
Some proofs of the results in this section are postponed to Section \ref{sec-aux}.
In Section \ref{sec-main}, we present our main results on causal Wasserstein DRO sensitivities including examples, and discuss their possible extensions.
A unified framework of the proofs is presented in Section \ref{sec-proof}.
In Sections \ref{sec-discrete}, \ref{sec-hyperbolic},and \ref{sec-parabolic}, we derive the discrete-time sensitivity, the continuous-time sensitivity under hyperbolic scaling, and the continuous-time sensitivity under parabolic scaling, respectively.

\section{Preliminaries}
\label{sec-pre}
\subsection{Notations}
For a  Polish space \(\cX\), we equip it with its Borel \(\sigma\)-algebra \(\cB(\cX)\).
Let \(\scrP(\cX)\) be the space of Borel probability measures on \(\cX\) equipped with its weak topology.
Given \(\mu\in\scrP(\cX)\) and a \(\sigma\)-algebra \(\cF\subseteq\cB(\cX)\), we denote the completion of \(\cF\) under \(\mu\)  by \(\prescript{\mu}{}\cF\).

Let \(\cX\) and \(\cY\) be two Polish spaces.
Given \(\mu\in\scrP(\cX)\) and \(\nu\in\scrP(\cY)\),  the set of couplings between \(\mu\) and \(\nu\) is defined as
\begin{equation*}
    \Pi(\mu,\nu):=\{\pi\in\scrP(\cX\times\cY):\pi(\cdot\times\cY)=\mu(\cdot)\text{ and } \pi(\cX\times \cdot)=\nu(\cdot)\}.
\end{equation*}
We also denote the set of couplings with a fixed first marginal \(\mu\) by
\begin{equation*}
    \Pi(\mu,*):=\{\pi\in\scrP(\cX\times\cY):\pi(\cdot\times\cY)=\mu(\cdot)\}.
\end{equation*}
Given \(\mu\in\scrP(\cX)\) and a measurable map \(\Phi:\cX\to\cY\), we define the \emph{pushforward} map \(\Phi_{\#}:\scrP(\cX)\to\scrP(\cY)\) by
\begin{equation*}
    \Phi_{\#}\mu:=\mu\circ \Phi^{-1} \quad\text{for any } \mu\in\scrP(\cX).
\end{equation*}
For any \(\pi\in \scrP(\cX\times \cY)\), we write
\begin{equation*}
    \pi(\md x,\md y)=\pi(\md x) \pi_{y}(x,\md y),
\end{equation*}
where  \(\pi_{y}\) is the Borel regular disintegration kernel.

Throughout the paper, we fix \(N,d\geq 1\), \(T>0\), \(p>1\) and let \(q=p/(p-1)\). We use the generic notation \(\langle\cdot,\cdot\rangle\) for the scalar product, where the space is clear from the context. We consider stochastic processes taking value in \(\bbR^{d}\) and recall that we use the notation \(|\cdot|\) for the \(l_{p}\)-norm on \(\bbR^{d}\), and \(|\cdot|_{*}\) for the \(l_{q}\)-norm on \(\bbR^{d}\).
In discrete-time, we take time index \(I=\{0,1,\dots,N\}\) and refer to the canonical path space as \(\cX=\{0\}\times(\bbR^{d})^{N}\) equipped with its natural filtration;
in continuous-time, we take time index \(I=[0,T]\) and refer to the canonical path space as \(\cX=C_{0}([0,T];\bbR^{d})\) equipped with the right continuous augmentation of its natural filtration.
Let \(W_{0}^{1,p}\) be the Sobolev subspace of \(C_{0}\) equipped with the norm
\begin{equation*}
    \|x\|_{W_{0}^{1,p}}=\|\dot{x}\|_{L^{p}}.
\end{equation*}

Given a  filtered probability space, \((\Omega,\cF,(\cF_{t})_{t\in I},P)\) and a measurable (not necessary adapted) process \(A:\Omega \to \cX\), we denote its natural filtration by \((\cF_{t}^{A})_{t\in I}\) where \(\cF_{t}^{A}=\sigma(A_{s}: s\in[0,t] )\).
In discrete time we write
\begin{equation*}
    \|A\|_{L^{q}(\mu)}=\lb(\sum_{n=1}^{N}\E_{\mu}\lb[|A_{n}|_{*}^{q}\rb]\rb)^{1/q};
\end{equation*}
and in continuous time we write
\begin{equation*}
    \|A\|_{L^{q}(\mu)}=\lb(\int_{0}^{T}\E_{\mu}\lb[|A_{t}|_{*}^{q}\rb]\md t\rb)^{1/q}.
\end{equation*}
For \(A\) with \(\E_{P}[\|A\|_{\infty}]<\infty\), we denote the optional projection of \(A\) by \(\op A\) which is the  unique optional process (up to a \(P\)-null set) such that for any bounded optional stopping time \(\tau\)
\begin{equation*}
    \op A_{\tau}=\E_{P}[A_{\tau}|\cF_{\tau}].
\end{equation*}
We denote the predictable projection of \(A\) by \(\pp A\) which is the  unique process (up to a \(P\)-null set) such that for any bounded predictable stopping time \(\tau\)
\begin{equation*}
    \pp A_{\tau}=\E_{P}[A_{\tau}|\cF_{\tau-}].
\end{equation*}
Note that in discrete time there are no issues with the pathwise regularity, and the projections are simply given by
\begin{equation*}
    \op X_{n}=\E[X_{n}|\cF_{n}] \quad\text{and}\quad \pp X_{n}=\E[X_{n}|\cF_{n-1}].
\end{equation*}

For a multidimensional martingale \(M\), we denote \(\lt M\rt\) its matrix-valued quadratic variation process and \([M]\) its trace, which we refer to as the scalar quadratic variation process. In the sequel, we use \emph{regular martingale} to refer to a continuous-time square-integrable continuous martingale \(M\) with absolutely continuous quadratic variation.
By \(\scrM(\cX)\) we denote the space of regular martingale measures on \(\cX\).
By \(\la\cdot, \cdot\ra_{\bF}\) and \(\|\cdot\|_{\bF}\) we denote the Frobenius inner product and norm respectively.

We adopt Landau symbols \(o\) and \(O\).
By \(o(r)\) we denote a quantity bounded by \(l(r)r\)  with \(\lim_{r\to 0}l(r)=0\), and by \(O(r)\) we denote a quantity bounded by \(l(r)r\) with \(\lim_{r\to 0}l(r)<\infty\).
We stress that \(l(r)\) is deterministic and independent of the underlying probability measure.

\subsection{Causal optimal transport}
Casual transport is an analogy of the classical transport in a dynamic context with the constraint that the transport plan is conditionally independent of the future given the current information.
Such a constraint leads to a finer ambiguity set than the classical optimal transport and can capture the temporal structure between stochastic processes.

Let \((\cX,\cF,(\cF_{t})_{t\in I})\) be the canonical path space, and \(X\) the canonical process. We follow the general \textit{causal} optimal transport framework introduced by \cite{lassalleCausalTransportPlans2018,acciaioCausalOptimalTransport2020}.
Let \((\cY,\cG,(\cG_{t})_{t\in I})\) be a copy of  \((\cX,\cF,(\cF_{t})_{t\in I})\).
\begin{defn}[\cite{acciaioCausalOptimalTransport2020}]
    Given \(\mu\in\scrP(\cX)\) and \(\nu\in\scrP(\cY)\), we say a coupling \(\pi\in\Pi(\mu,\nu)\) is \textit{causal} from \(\mu\) to \(\nu\) if for any \(t\in I\) and \(U\in \cG_{t}\)
    \begin{equation}
        \label{def-caus}
        \cX\ni x\mapsto \pi_{y}(x,U)\in\bbR
    \end{equation}
    is \(\prescript{\mu}{}\cF_{t}\)--measurable.
    We say \(\pi\in\Pi(\mu,\nu)\) is bicausal if it is causal and \([(x,y)\mapsto (y,x)]_{\#}\pi\in \Pi_{\c}(\nu,\mu)\).
    We denote the set of causal (bicausal) couplings from \(\mu\) to \(\nu\) by \(\Pi_{\c}(\mu,\nu)\) \((\Pi_{\bc}(\mu,\nu))\).
\end{defn}

The motivation \citep{backhoff-veraguasCausalTransportDiscrete2017} behind the  above definition is the following: if we assume a \textit{causal} transport plan \(\pi\) is generated by a Monge map \(\Phi:\cX\to\cY\) through \(\pi=(\Id,\Phi)_{\#} \mu\), then \(\Phi\) is \textit{non-anticipative} (adapted) in the sense that
\begin{equation*}
    \Phi^{-1}(\cG_{t})\subseteq \cF_{t}.
\end{equation*}
Let cost \(c:\cX\times\cY\to \bbR\).
For any \(\mu,\nu\in\scrP_{p}(\cX)\), the optimal causal transport cost is given by
\begin{equation}
    d_{\c}(\mu,\nu):=\inf_{\pi\in \Pi_{\c}(\mu,\nu)}\E_{\pi}[c(X,Y)].
\end{equation}
Similarly, the optimal bicausal transport cost \(d_{\bc}(\mu,\nu)\) is defined by replacing \(\Pi_{\c}(\mu,\nu)\) with \(\Pi_{\bc}(\mu,\nu)\).
The following is a direct consequence of \citet[Remark 2.3 (4)]{acciaioCausalOptimalTransport2020}.
\begin{prop}
    \label{prop-mart}
    Let \(\mu,\nu\in \scrM(\cX)\). Then \(\pi\in \Pi_{\bc}(\mu,\nu)\) implies that \((X,Y)\) is an \((\cF_{t}\otimes \cG_{t})_{t\in I}\)--martingale. Suppose additionally that \(X\) and \(Y\) have the martingale representation property under \(\mu\) and \(\nu\) respectively. Then  \((X,Y)\) is an \((\cF_{t}\otimes\cG_{t})_{t\in I}\)--martingale under \(\pi\in \Pi(\mu,\nu)\) implies that \(\pi\) is bicausal.
\end{prop}

\begin{rmk}
    In general, the above proposition does not hold for causal couplings.
    For example, we consider L\'evy's noncanonical representation of the Brownian motion: \(X_{t}=\int_{0}^{t}\{3-12 (s/t)+10(s/t)^{2}\}\md B_{s}\).
    This induces a causal coupling between the Wiener measures and it is not bicausal.
    In particular, \(X\) is not an \((\cF_t^X\vee \cF_t^B)_{t\in I}\)--martingale.
\end{rmk}
In discrete-time setting, \cite[Lemma 3.1]{bartlSensitivityMultiperiodOptimization2022}  shows bicausal couplings \(\Pi_{\bc}(\mu,*)\) are dense in the set of causal couplings \(\Pi_{\c}(\mu,*)\).
The proof immediately adapts if instead of couplings we consider transport maps.
Here, we observe that both also extend to the setting under a martingale constraint.
We state the results for maps as this is the version we use later on.
The proof is deferred to Section \ref{sec-aux}.
\begin{prop}
    \label{prop-dense}
    Let \(\cX=\{0\}\times(\bbR^{d})^{N}\), \(\mu\in\scrP(\cX)\) and \(\pi=(\Id,\Phi)_{\#} \mu\in \Pi_{\c}(\mu,*)\).
    Then for any \(\varepsilon>0\), there exists \(\Phi^{\varepsilon}\) such that \(\|\Phi-\Phi^{\varepsilon}\|_{\infty}<\varepsilon\) and \((\Id,\Phi^{\varepsilon})_{\#}\mu\in \Pi_{\bc}(\mu,*)\).
    Moreover, if \(\mu, \Phi_{\#}\mu\in \scrM(\cX)\) then $\Phi^{\varepsilon}$ can be taken such that \(\Phi^{\varepsilon}_{\#}\mu\in \scrM(\cX)\).
\end{prop}

\section{Some tools from stochastic calculus}
\label{sec-tools}
\subsection{Classical Malliavin derivative}
We follow \citet{nualart2006malliavin} to give a brief introduction to the (classical) Malliavin calculus. We then introduce the notion of \emph{pathwise} Malliavin derivative. It arises naturally as the limit of discrete objects, and we show it coincides with the classical version on the intersection of their domains. We believe such a pathwise approach to Malliavin calculus can unlock many interesting results. It links with the functional It\^o calculus of \cite{dupire2019functional,contfournie:13} and will be key to extending our sensitivity results in continuous time to reference measures beyond the Brownian case. We plan to pursue this direction of research in a future paper, see also Section \ref{sec:further}.

Let \(X\) be a \(d\)-dim Brownian motion on a filtered probability space \((\Omega,\cF,(\cF_{t})_{t\in I},P)\).
For a smooth cylindrical random variable \(F=f(\int_{0}^{T} \la h_{t}, \md X_{t}\ra)\) where \(f\in C^{1}_{b}\) and \(h\in L^{2}([0,T],\bbR^{d})\), its Malliavin derivative \(\MD F\) is given by
\begin{equation*}
    \MD_{t}F:= f'\lb(\int_{0}^{T} \la h_{s},\md X_{s}\ra \rb) h_{t} .
\end{equation*}
It is well-known that \(\MD\) is closable on \(L^{2}(P,\cF_{T}^{X})\).
Hence, we do not distinguish \(\MD\) from its closure and denote its domain by \(\MD^{1,2}\).
The predictable projection of the Malliavin derivative solves the martingale representation problem.
\begin{thm}[Clark--Ocone formula]
    Assume \(Z\in\MD^{1,2}\).
    Then we have
    \begin{equation*}
        Z=\E_{P}[Z]+\int_{0}^{T}\la \pp\MD_{t}Z,\md X_{t}\ra.
    \end{equation*}
\end{thm}
Malliavin derivative can be interpreted as a `gradient' operator on the `tangent space', the Cameron--Martin space \(W_{0}^{1,2}\).
\begin{prop}
    Assume \(f:\cX\to\bbR\) and \(f(X)\in \MD^{1,2}\).
    Then for any \(\eta\in W_{0}^{1,2}\), we have
    \begin{equation}\label{eq:classMD_char}
        \lim_{\varepsilon\to 0} \frac{f(X+\varepsilon \eta )-f(X)}{\varepsilon}= \int_0^T \la\MD_t f(X),\dot{\eta}\ra \md t \quad P\text{--a.s.}.
    \end{equation}
\end{prop}

\subsection{Pathwise Malliavin derivative}
To introduce pathwise Malliavin derivative, we start with the discrete-time setup. Recall that \(\Delta:\cX\to\cX\) is the increment map given by
\begin{equation*}
    \Delta x= (0,\Delta x_{1}, \dots, \Delta x_{N}) := (0,x_{1},x_{2}-x_{1},\dots,x_{N}-x_{N-1}).
\end{equation*}
Then the \emph{pathwise} Malliavin derivative \(\mD=(\mD_{1},\dots,\mD_{N})\) is defined as the pullback of \(\nabla=(\partial_{1},\dots,\partial_{N})\) under \(\Delta\), i.e., for any smooth \(f\)
\begin{equation}\label{eq:def_discrete_MD_1}
    \la \mD f (x), y\ra :=\la \nabla f (x), \Delta^{-1} y \ra=\la (\Delta^{-1})^{\star}\nabla f (x),  y \ra,
\end{equation}
where \((\Delta^{-1})^{\star}\) is the adjoint operator of \(\Delta^{-1}\). Expressed explicitly, this gives
\begin{equation}\label{eq:def_discrete_MD}
    \mD_{n} = \sum_{k=n}^{N}\partial_{k},\quad 1\leq n\leq N.
\end{equation}
Here, \(\bbD f\) is an analogue of Malliavin derivative in the sense that for any deterministic path \(x,y\in\cX\)
\begin{equation*}
    \lim_{\varepsilon\to 0} \frac{f(x+\varepsilon y )-f(x)}{\varepsilon}= \la\bbD f(x),\Delta y\ra.
\end{equation*}

In continuous-time, while we focus on stochastic processes with continuous path, it is natural to first define \emph{pathwise} Malliavin derivative for functionals on the c\`adl\`ag path space \(D([0,T];\bbR^{d})\).
This approach closely aligns with the functional It\^o calculus from \citet{contfournie:13}.
Proofs of results in this section are deferred to Section \ref{sec-aux}.
\begin{defn}[Pathwise Malliavin derivative]
    Let \(f:D([0,T];\bbR^{d})\to \bbR\) be a functional on the c\`adl\`ag path space. We say that \(f\) is \emph{pathwise Malliavin differentiable} if there exists \(\mD f:D([0,T];\bbR^{d})\times [0,T] \to \bbR^d\) such that
    \begin{equation*}
        \la\mD_{t} f(\omega), e \ra = \lim_{\varepsilon\to 0} \frac{f(\omega+ \varepsilon e \1_{[t,T]})-f(\omega)}{\varepsilon},\quad \forall   \omega\in D([0,T];\bbR^{d}),\ e\in \bbR^{d}.
    \end{equation*}
    We call \(\mD f=(\mD_{t} f)_{t\in[0,T]}\) the \textit{pathwise Malliavin derivative} of \(f\).
\end{defn}

\begin{defn}[Left limit]
    For a functional \(F:[0,T]\times D([0,T];\bbR^{d})\to\bbR\), we say \(F\) has a left limit at \(t\in[0,T]\) if for any \(\omega^{n}\) converging uniformly to \(\omega\) and \(t^{n}\) increasing to \(t\) the following limit exists
    \[\lim_{n\to\infty}F(t^{n},\omega^{n})=F(t-,\omega).\]
\end{defn}

\begin{defn}[Boundness preservation]
    We say a functional \(F:[0,T]\times D([0,T];\bbR^{d})\to\bbR\) is boundedness preserving if for any \(K\)
    \[\sup_{t\in[0,T]}\sup_{\|\omega\|_{\infty}<K}|F(t,\omega)|<\infty.\]
\end{defn}

\begin{prop}
    \label{prop-derivative}
    Assume   \(f\) and \(\mD_{t}f\), \(t\in I\), are all continuous with respect to the uniform topology.
    Then, for any simple step function  \(\eta=\sum_{k=1}^{n} e_{k}\1_{[t_{k},1]}\), we have
    \begin{equation*}
        \lim_{\varepsilon\to 0}\frac{f(\omega+\varepsilon \eta)-f(\omega)}{\varepsilon}= \sum_{k=1}^{n} \la\mD_{t_{k}} f(\omega), e_{k} \ra.
    \end{equation*}
    If further \(\mD f\) has a left limit for all \(t\in [0,T]\) and is boundedness preserving then, for any path \(\eta\in AC_{0}([0,T];\bbR^{d})\), we have
    \begin{equation}\label{eq:pathMD_char}
        \lim_{\varepsilon\to 0}\frac{f(\omega+\varepsilon \eta)-f(\omega)}{\varepsilon}= \int_{0}^{T}\la\mD_{t} f(\omega), \dot{\eta}_{t} \ra\md t.
    \end{equation}
\end{prop}
We note that \eqref{eq:pathMD_char} fully characterizes (up to a Lebesgue null set) the left limit of the Malliavin derivative.
This, in particular, offers us a natural way to define the pathwise Malliavin derivative for a functional defined on the continuous path space \(f:C_{0}([0,T];\bbR^{d})\to\bbR\).

\begin{defn}\label{def:Db1}
    We denote \(\mD^{1}_{b}\) the space of functionals \(f:C_{0}([0,T];\bbR^{d})\to\bbR\) which admit an extension \(\tilde{f}\) to \(D([0,T];\bbR^{d})\) satisfying assumptions in Proposition \ref{prop-derivative}, and for such $f$ we set
    \begin{equation*}
        \mD_{t}f(\omega) :=\mD_{t}\tilde{f}(\omega),\quad t\in [0,T],\ \omega\in C_{0}([0,T];\bbR^{d}).
    \end{equation*}
\end{defn}
We remark that \(\mD f\) is measurable and does not depend on the choice of the extension \(\tilde{f}\) as the right-hand side only depends on the values of \(\tilde{f}\) on the continuous path space.

\begin{exam}
    We give a few examples of functionals \(f\in\mD^{1}_{b}\).
    \begin{enumerate}
        \item \(f(\omega)=g(\omega_{t_{1}},\dots,\omega_{t_{n}})\), where \(g\in C^{1}_{b}\).
              This gives \(\mD_{t}f(\omega)=\sum_{k=1}^{n}  \partial_{k}g(\omega_{t_{1}},\dots,\omega_{t_{n}})\1_{[0,t_{k}]}\).
        \item \(f(\omega)=\int_{0}^{T}g(\omega_{t})\theta(\md t)\) where \(g\in C^{1}_{b}\) and \(\theta\) is a finite measure.
              This gives \(\mD_{t}f(\omega)=\int_{t}^{T}\nabla g (\omega_{s})\theta(\md s)\).
    \end{enumerate}
\end{exam}

The following property shows that the classical and the \emph{pathwise} Malliavin derivatives agree on their common domain.
\begin{prop}\label{prop:pathMDisclassicalMD}
    Let \(\cX=C_{0}([0,T];\bbR^{d})\), \(\mu\) the Wiener measure, and \(X\) be the canonical process.
    Assume \(f\in\mD^{1}_{b}\) and \(f(X)\in\MD^{1,2}\).
    Then, we have
    \begin{equation*}
        (\mD_{t}f) (X)=\MD_{t} f(X)  \quad  \md \mu\otimes\md t\text{--a.e.}.
    \end{equation*}
\end{prop}

\subsection{Forward integral against regular martingales}
\label{sec-Skoro}
Let \(X\) be a \(d\)-dim Brownian motion on a filtered probability space \((\Omega,\cF,(\cF_{t})_{t\in I},P)\).
The forward integration was introduced as an anticipative extension to the It\^o integral, and it plays a crucial role in our sensitivity analysis.
We first recall its definition.
\begin{defn}[\citet{russo1993forward}]
    Let  \(A,B\) be two measurable stochastic processes (not necessary adapted).
    The forward integral is given by
    \begin{equation*}
        \int_{0}^{T}\la A_{t},\md^{-}B_{t}\ra=\lim_{\varepsilon\to 0}\int_{0}^{T}\frac{1}{\varepsilon}\la A_{t},B_{(t+\varepsilon)\wedge T}-B_{t}\ra\md t,
    \end{equation*}
    if the right hand side limit exists in probability.
    We say \(A\) is \(B\)  forward--\(\gamma\) integrable if the above limit exists in \(L^{\gamma}\).
    A family \(\{A_{\lambda}\}_{\lambda\in \Lambda}\) of \(B\)  forward--\(\gamma\) integrable processes is said to be \(B\) uniformly forward--\(\gamma\) integrable if
    \begin{equation*}
        \sup_{\lambda\in \Lambda} \lb|\int_{0}^{T}\la A_{t,\lambda}, \md^{-} B_{t}\ra -\int_{0}^{T}\frac{1}{\varepsilon}\la A_{t,\lambda},B_{(t+\varepsilon)\wedge T}-B_{t}\ra\md t\rb|
    \end{equation*}
    converges to 0 in \(L^{\gamma}\) as \(\varepsilon\) goes to 0.
\end{defn}

We focus on the case where the integrator of the forward integral is a regular martingale.
Unlike It\^o integral, the expectation of a forward integral can be nonzero as a consequence of the anticipative nature of the integrand.
\begin{prop}
    \label{prop-fw}
    Let \(M\) be a regular martingale.
    We assume \(\Phi_{t}\in \cF_{T}^{X}\) for any \(t\in I\), and \(\Phi\) is \(M\) forward--1 integrable.
    If the right limit \(\lim_{s\to t+}\MD_{s}\Phi_{t}=\MD_{t+}\Phi_{t}\) converges in \(L^{2}\) and \(\sup_{s,t\in[0,T]}\|\MD_{s}\Phi_{t}\|_{\bF}\in~L^{2}\),
    then we have
    \begin{equation*}
        \E_{P}\lb[\int_{0}^{T}\la \Phi_{t},\md^{-}M_{t}\ra\rb]=\E_{P}\lb[\int_{0}^{T}\la \MD_{t+}\Phi_{t},\md \lt X, M \rt_{t}\ra_{\bF}\rb],
    \end{equation*}
\end{prop}
\begin{proof}
    By definition of the forward-1 integrability and \(L^1\) convergence, we have
    \begin{align*}
        \E_{P}\lb[\int_{0}^{T}\la \Phi_{t},\md^{-}M_{t}\ra\rb]=\lim_{\varepsilon\to 0}\E_{P}\lb[\int_{0}^{T}\frac{1}{\varepsilon}\la \Phi_{t},M_{(t+\varepsilon)\wedge 1}-M_{t}\ra\md t\rb].
    \end{align*}
    By applying Clark--Ocone formula to \(\Phi_{t}\), we derive
    \begin{align*}
        \E_{P}\lb[\int_{0}^{T}\la \Phi_{t},\md^{-}M_{t}\ra\rb] & =\lim_{\varepsilon\to 0}\E_{P}\lb[\int_{0}^{T}\frac{1}{\varepsilon}\lb\la\int_{0}^{T} \pp\MD_{s}^{\intercal}\Phi_{t}\md X_{s},\int_{t}^{(t+\varepsilon)\wedge 1}\md M_{s}\rb\ra \md t\rb] \\
                                                               & = \lim_{\varepsilon\to 0} \E_{P} \lb[\int_{0}^{T}\frac{1}{\varepsilon}\int_{t}^{(t+\varepsilon)\wedge 1}\la\pp \MD_{s}\Phi_{t},\md\lt X, M \rt_{s} \ra_{\bF}\md t\rb]                     \\
                                                               & =\E_{P}\lb[\int_{0}^{T}\la \pp\MD_{t+}\Phi_{t},\md \lt X, M \rt_{t}\ra_{\bF}\rb],
    \end{align*}
    where  the second line follows from the It\^o isometry and Fubini theorem, and the last line follows from the dominated convergence theorem  and \(\sup_{s,t\in[0,T]}\|\MD_{s}\Phi_{t}\|_{\bF}\in L^{2}\).
    We conclude the proof by noticing \(\lt X,M \rt\) is a predictable process.
\end{proof}

We give now a stochastic Fubini theorem which details the correction term arising from an  interchange between a forward and an It\^o integral. To the best of our knowledge, this is a novel result which we believe is of independent interest. Its proof is deferred to Section \ref{sec-aux}.
\begin{thm}[Stochastic Fubini theorem]
    \label{thm-fubini}
    Let \(\gamma\in[1,2)\).
    Let \(\Psi:I\times I\times \Omega\to \bbR^{d\times d}\), \(M\) be a regular martingale.
    We assume that \(\Psi_{s,\cdot}\) is predictable for any \(s\in I\) and satisfies the following conditions:
    \begin{enumerate}
        \item \(\{\Psi_{\cdot,t}\}_{t\in[0,T]}\) is \(M_{\cdot}\) uniformly forward--\(\gamma\) integrable,
        \item    \(
              \sup_{s,t\in[0,T]}\|\Psi_{s,t}\|_{\bF}\in L^{2\gamma/(2-\gamma)},
              \)
        \item \(\lim_{t\to s+}\E_{P}[\|\Psi_{s,t}-\Psi_{t,t}\|_{\bF}^{2}]=0\) for any \(s\in[0,T]\).
    \end{enumerate}

    Then, we have \(\int_{\cdot}^{T} \Psi_{\cdot,t}^{\intercal}\md X_{t}\) is \(M_{\cdot}\) forward--\(\gamma\) integrable.
    Moreover,
    \begin{equation*}
        \int_{0}^{T}\lb\la\int_{s}^{T} \Psi_{s,t}^{\intercal}\md X_{t},\md^{-} M_{s}\rb\ra=   \int_{0}^{T}\lb\la\int_{0}^{t} \Psi_{s,t} \md^{-} M_{s}, \md X_{t}\rb\ra + \int_{0}^{T}\lb\la \Psi_{t,t} ,\md \lt X, M\rt_{t}\rb\ra_{\bF} .
    \end{equation*}

\end{thm}

\section{Sensitivity of (bi)causal DRO}
\label{sec-main}

With the tools developed so far, we can now give rigorous statements of our main results on the sensitivity of (bi)causal DRO problems to model uncertainty. We recall that $p>1$ is fixed, $q=\frac{p}{p-1}$, and the parametrized penalty is given by \(L_{\delta}(\cdot)=\delta L(\cdot/\delta)\).
We impose the following assumption on  \(L\).
The growth condition ensures that when \(\delta\) goes to 0, the adversarial distribution will converge to the reference model \(\mu\). We stress that the proofs offer direct characterizations of the first-order optimal adversarial model perturbations. We make this explicit only for the first theorem, see Remark \ref{rk:adversarialdirection}. As highlighted in \citep{bai2024wasserstein} this can be as important as the sensitivity computation itself.

\begin{asmp}
    \label{asmp-pen}
    We assume that \(L:[0,+\infty)\to[0,+\infty]\) is continuous, non-decreasing, and satisfies
    \begin{equation*}
        L(0)=0 \quad \text{and}\quad \liminf_{u\to\infty}\frac{L(u)}{u^{p}}=+\infty.
    \end{equation*}
\end{asmp}
We write the convex conjugate of \(L\) as \(L^{*}\) given by
\begin{equation*}
    L^{*}(v)=\sup_{ u \geq 0} \{uv - L(u)\}.
\end{equation*}

\subsection{Discrete-time results}
We start with the discrete time setting. We state and discuss the results, with the proofs deferred to Section \ref{sec-discrete}.
Let \(d_{\c}\) and \(d_{\bc}\) be the causal and bicausal transport cost induced by the cost
\begin{equation*}
    c_{N}(x,y)=\sum_{n=1}^{N}|\Delta x_{n}- \Delta y_{n}|^{p},
\end{equation*}
and recall \(\mD\) is the discrete Malliavin derivative in \eqref{eq:def_discrete_MD}. We consider
\begin{equation}\label{eq:Vdef}
    V(\delta)=\sup_{\nu\in \scrP(\cX)}\lb\{\E_{\nu}[f(X)]-L_{\delta}(d_{\c}(\mu,\nu)^{1/p})\rb\}.
\end{equation}

\begin{asmp}
    \label{asmp-loss}
    We assume  that \(f:\cX\to \bbR\)  is continuously differentiable.
    Moreover, \(\mD f\) satisfies for any \(x\in \cX\)
    \begin{equation}
        \label{eqn-grad}
        |\mD f(x)|\leq C (1+|x|^{p-1}).
    \end{equation}
\end{asmp}

\begin{thm}
    \label{thm-sens}
    Under Assumptions \ref{asmp-pen} and \ref{asmp-loss}, we have
    \begin{equation*}
        V(\delta)=V(0)+\Upsilon\delta +o(\delta),
    \end{equation*}
    where
    \begin{equation*}
        \Upsilon:=\lim_{\delta\to 0}\frac{1}{\delta}(V(\delta)-V(0))=L^{*}\lb(\E_{\mu}\lb[\sum_{n=1}^{N}\bigl|\E_{\mu}[\mD_{n} f(X)|\cF_{n}]\bigr|_{*}^{q}\rb]^{1/q}\rb)=L^{*}\lb(\|\op\mD f\|_{L^{q}(\mu)}\rb).
    \end{equation*}
\end{thm}

\begin{rmk}
    \label{rk:adversarialdirection}
    The first-order optimal adversarial model $\nu_\delta$ which approximates \eqref{eq:Vdef} for small $\delta$ is given explicitly in the proof as \(\nu_{\delta}=(\Id + u \delta \Delta^{-1}\circ \Phi)_{\#}\mu\), where $\Phi$ is given in \eqref{eq:Phindef_discrete} and $u$ is such that $\Upsilon = u\|\op\mD f\|_{L^{q}(\mu)} - L(u)$.
\end{rmk}
We now turn to the problem under martingale constraint. Let $\mu\in\scrM(\cX)$ and consider
\begin{equation*}
    V_{\M}(\delta)=\sup_{\nu\in \scrM(\cX)}\lb\{\E_{\nu}[f(X)]-L_{\delta}(d_{\bc}(\mu,\nu)^{1/p})\rb\}.
\end{equation*}
\begin{rmk}
    We note that in the discrete-time setting the bicausal and the causal penalizations are interchangeable by Proposition \ref{prop-dense}.
    We use different penalizations for $V(\delta)$ and $V_{\M}(\delta)$ to be consistent with their continuous-time counterparts.
\end{rmk}

\begin{thm}
    \label{thm-sens-mart}
    Let \(\scrH\) denote the set of predictable processes $h$ with $\|h\|_{L^q(\mu)}<\infty.$
    Under Assumptions \ref{asmp-pen} and \ref{asmp-loss}, we have
    \begin{equation*}
        V_{\M}(\delta)=V_{\M}(0)+\Upsilon_{\M}\delta +o(\delta),
    \end{equation*}
    where
    \begin{equation*}
        \Upsilon_{\M}=L^{*}\lb( \inf_{h\in\scrH}\E_{\mu}\lb[\sum_{n=1}^{N}\bigl|\E_{\mu}[\bbD_{n} f(X)|\cF_{n}]-h_{n}\bigr|_{*}^{q}\rb]^{1/q}\rb)=L^{*}\lb(\inf_{h\in\scrH}\|\op\mD f-h\|_{L^{q}(\mu)}\rb).
    \end{equation*}
    In particular, if \(p=2\), we obtain
    \begin{equation*}
        \Upsilon_{\M}=L^{*}\lb(\|\op\mD f-\pp\mD f\|_{L^{2}(\mu)}\rb).
    \end{equation*}
\end{thm}

\begin{rmk}\label{rmk:linkwithDaniJohannes}
    We note that a variant of the unconstrained case, Theorem \ref{thm-sens}, was established independently in \citet{bartlSensitivityMultiperiodOptimization2022} with \(L=+\infty\1_{(1,\infty)}\) and \(c(x,y)=\sum_{n=1}^{N}|x_{n}-y_{n}|^{p}\).
    For the martingale constrained case,  Theorem \ref{thm-sens-mart} intersects with the results in \citet{sauldubois24First}, where the authors independently derived the sensitivity for a two-step case with the same choice of \(L\) and \(c\) as in \citet{bartlSensitivityMultiperiodOptimization2022}. We emphasize the difference in the choice of the cost function: while equivalent in discrete time, our choice allows obtaining continuous time results as limits of the discrete time results.
    For the martingale constraint problem in Theorem~\ref{thm-sens-mart}, we can formally rewrite it as an unconstrained problem with a Lagrange multiplier:
    \begin{equation*}
        V_{\M}(\delta)= \sup_{\nu\in\scrP(\cX)}\inf_{h\in\scrH_b}\lb\{\E_{\nu}[f(X)-h\circ X]-L_{\delta}(d_{\bc}(\mu,\nu)^{1/p})\rb\},
    \end{equation*}
    where \(\scrH_b\) is the set of bounded predictable processes and \(h\circ X\) denotes the discrete stochastic integral.
    This formulation aligns with \citet[Theorem 2.4]{bartlSensitivityMultiperiodOptimization2022} by taking \(L=+\infty\1_{(1,\infty)}\) and \(c(x,y)=|x-y|^{p}\).
    However, \citet[Assumption 2.4]{bartlSensitivityMultiperiodOptimization2022} essentially imposes  the uniqueness of the optimizer \(h^{*}\) for
    \begin{equation*}
        \inf_{h\in\scrH_b} \E_{\mu}[f(X)-h\circ X],
    \end{equation*}
    which is of sharp contrast to the martingale constraint problem where any \(h\in\scrH_b\) is an optimizer.

\end{rmk}

\begin{exam}\label{ex:Greek}
    In the context of derivatives pricing in mathematical finance, the (bi)causal DRO sensitivity under martingale constraint can be viewed as a \emph{nonparametric} Greek. It captures the sensitivity of the option price to model uncertainty. This was first observed in \cite{bartlSensitivityAnalysisWasserstein2021} in the context of perturbations of the distribution of the underlying stock price process at a given time (the maturity). Having derived sensitivities in a dynamic context, we can consider perturbations of the actual model for the price process.

    We consider $d=1$ and a discrete-monitored Asian option whose payoff is given by
    \begin{equation}\label{eq:Asianpayoff}
        f(X)=\max\lb\{0,\bar{X}-K\rb\} \quad \text{with} \quad \bar{X}=\frac{1}{N+1}\sum_{n=0}^{N}X_{n}.
    \end{equation}
    Let \(\mu\in \scrM(\cX)\) be the reference risk-neutral (pricing) measure.
    Without loss of generality, we assume \(\mu\) is centreed, otherwise we can always shift the market by a constant and absorb it into \(K\).
    Notice that
    \begin{equation*}
        \mD_{n} f(X)=\frac{N+1-n}{N+1}\1_{\{\bar{X}\geq K\}}.
    \end{equation*}
    For simplicity, we take \(p=2\) and \(L=+\infty\1_{(1,\infty)}\).
    By Theorem \ref{thm-sens-mart}, we derive the nonparametric `Greek' of the Asian option as
    \begin{align}
        \Upsilon_{\M} & = \lb(\E_{\mu}\lb[\sum_{n=1}^{N} \lb|\E_{\mu}[\mD_{n}f(X)|\cF_{n}]-\E_{\mu}[\mD_{n}f(X)|\cF_{n-1}]\rb|^{2}\rb]\rb)^{1/2}    \nonumber                                      \\
                      & =\lb(\E_{\mu} \lb[ \sum_{n=1}^{N}\frac{(N+1-n)^{2}}{(N+1)^2}\lb|\mu(\bar{X}\geq K|\cF_{n})-\mu(\bar{X}\geq K|\cF_{n-1})\rb|^{2}\rb] \rb)^{1/2}.\label{eq:UpsilonMartAsian}
    \end{align}
    To compare this result with a parametric sensitivity, consider $\mu(\mathfrak{j})$ to be the distribution of a symmetric random walk with jump size $\pm \mathfrak{j}$, and set $\mu=\mu(1)$, the distribution of the simple symmetric random walk.
    The resulting sensitivity is displayed in Figure \ref{fig:asiandiscrete}.
    Whilst the parametric sensitivity captures the main risk, the nonparametric one dominates it, as expected.
    Inspecting the first-order optimal adversarial model perturbation, see \eqref{eq:Phindef_discrete}, reveals that it involves both the jump size and their symmetry being broken simultaneously.

    \begin{figure}
        \centering
        \includegraphics[width=10cm]{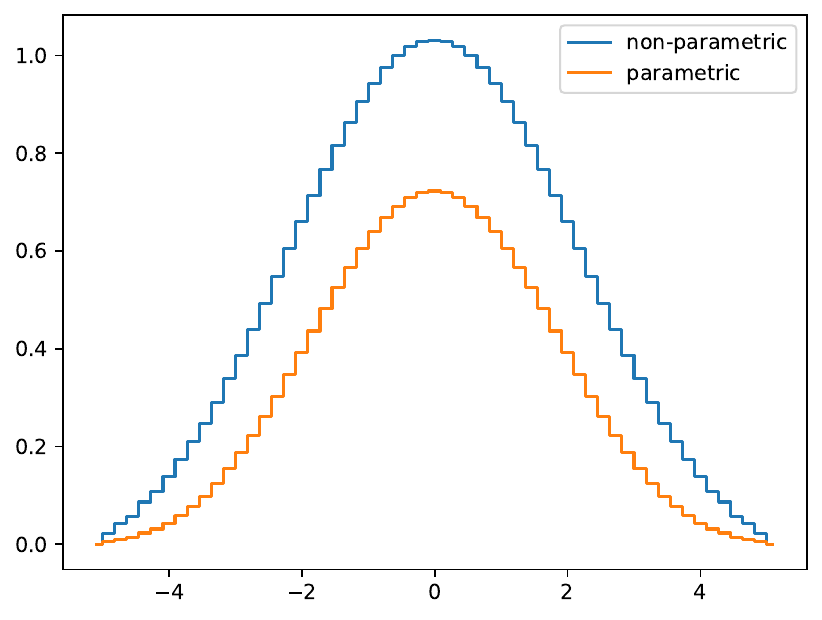}
        \caption{Comparison of sensitivities of the Asian option \eqref{eq:Asianpayoff} price under $\mu$, with $N=10$, as a function of strike $K$: parametric sensitivity with respect to the jump size $\mathfrak{j}$ and the non-parametric  $\Upsilon_{\M}$ given in \eqref{eq:UpsilonMartAsian}.
        }
        \label{fig:asiandiscrete}
    \end{figure}

\end{exam}

\subsection{Continuous-time results}
We switch now to the continuous time setting. As before, we state and discuss the main results, with proofs deferred to \ref{sec-continuous}.
We start with the  hyperbolic scaling where the transport cost is given by
\[
    c(\omega,\eta)= \limsup_{N\to\infty}\frac{N^{p-1}}{T^{p-1}} \sum_{n=1}^{N}|\Delta\omega_{nT/N}-\Delta \eta_{nT/N}|^{p}=  \left\{\begin{aligned}&\|\omega-\eta\|_{W^{1,p}_{0}}^{p}  \text{ if } \omega-\eta\in W_{0}^{1,p},\\ & +\infty  \qquad \quad \text{ elsewhere}.\end{aligned}\right.
\]

\begin{asmp}
    \label{asmp-cont}
    Function \(f:\cX\to\bbR\) is in \(\mD^{1}_{b}\) and satisfies
    \begin{itemize}
        \item \(|f(\omega)|\leq C (1+\|\omega\|_{\infty}^{p})\), \(\omega\in \cX\);
        \item \(\mD_{t} f\) is continuous, and \(|\mD_{t}f(\omega)|\leq C (1 +\|\omega\|_{\infty}^{p-1})\), $t\in [0,T]$.
    \end{itemize}
\end{asmp}
\begin{rmk} We recall that \(\mD^{1}_{b}\) was introduced in Definition \ref{def:Db1}.
    Note that we do not expect the growth of \(\mD_{t}f\) to imply a control on the growth of \(f\).
    This is because \(\mD_{t}f\) is only a directional derivative along a proper subspace of the tangent space.
\end{rmk}

\begin{thm}
    \label{thm-cont}
    Let $p>1$ and suppose Assumptions \ref{asmp-pen} and  \ref{asmp-cont} hold, and that  \(\mu\in \scrP(\cX)\) satisfies \(\E_{\mu}[\sup_{t\in[0,T]}|X_{t}|^{p}]<\infty\).
    Then, with \(V(\delta)\) given in \eqref{eq:Vdef}, we have
    \begin{equation*}
        V(\delta)=V(0)+\Upsilon\delta+o(\delta),
    \end{equation*}
    where
    \begin{equation*}
        \Upsilon=\lim_{\delta\to 0}\frac{1}{\delta}(V(\delta)-V(0))=L^{*}\lb(\|\op\mD f\|_{L^{q}(\mu)}\rb).
    \end{equation*}
\end{thm}
\begin{exam} \label{ex:merton}
    We consider Merton's \citep{merton:69} classical setting of a log investor maximizing the expected utility of their terminal wealth.
    The stock price process follows the standard Black-Scholes model, with \(S\) solving
    \[ dS_t = \zeta S_t dt + \sigma S_t dX_t, \]
    where \(X\) is a Brownian motion. The agent invests their wealth into the stock and a riskless asset which grows at a constant interest rate \(r\). Suppose their initial wealth is \(\kappa\) and let \(\theta_t\) denote the \emph{proportion} of their wealth invested in the risky asset at time \(t\), which is assumed to be \(\cF_t\)-measurable. Their wealth process \((K_t^\theta)_{t\in I}\) evolves according to
    \[ dK^\theta_t = (r+\lambda\theta_t\sigma)K^\theta_t dt + \sigma\theta_t K^\theta_t dX_t,\]
    where \(\lambda = (\zeta-r)/\sigma\), known as the market price of risk, is the key market parameter the investor has to estimate.
    Merton's problem of maximizing \(\bbE[\log(K^\theta_T)]\) over the choice of \(\theta\) is solved taking \(\theta_t = \lambda/\sigma\), and the resulting wealth satisfies
    \(K^*_T = \kappa \exp((r+\lambda^2/2)T + \lambda X_T).\)
    Agent's expected utility is given by \(V(0)= \bbE[\log(K^*_T)] = \log(\kappa) + (r+\lambda^2/2) T\) and its parametric sensitivity to \(\lambda\), which captures the agent welfare's sensitivity to the estimated parameter, is \(\frac{\partial}{\partial \lambda}V(0)= \lambda T\). The general sensitivity to model uncertainty, around \(\mu\) the Wiener measure, can be computed using Theorem \ref{thm-cont} for
    \[f(X)= \log(K^*_T) = \log(\kappa) + (r+\lambda^2/2)T+\lambda X_T.\] Taking \(p=2\) and  \(L=+\infty\1_{(1,\infty)}\), we obtain \(\Upsilon = \lambda \sqrt{T}\).
    We see that \(\Upsilon\) recovers the parametric sensitivity but with a different scaling in time. Indeed, $\sqrt{T}$ is the natural Brownian scaling in time and we kept the uncertainty penalty $L$ independent of time. If instead, we set \(L=+\infty\1_{(\sqrt{T},\infty)}\), introducing the natural Brownian scaling into the size of the uncertainty ball considered, then we obtain \(\Upsilon=\lambda T\), as before. Naturally, the parametric sensitivity only makes sense in the specific context of Black-Scholes price dynamics, and \(\Upsilon \) offers its natural non-parametric extension to general investment settings.
\end{exam}

\begin{exam}
    \label{exam-rough}
    We stress that our continuous-time sensitivity results can also go beyond the semi-martingale framework.
    Consider an objective given by a pathwise rough integral introduced in \citet{cont2019pathwise}.
    We fix a positive integer \(l\) and a sequence of partitions \(\mathfrak{p}=\{p_{1},p_{2},\dots\}\)
    where \(p_{n}=\{0=t_{0}<\cdots<t_{i_n}=T\}\).
    By \(\mathbb{V}_{2l}(\mathfrak{p})\), we denote the set of paths with finite \(2l\)-variation along \(\mathfrak{p}\) in the sense that
    \begin{equation*}
        \lim_{n\to\infty} \sum_{[t_{i},t_{i+1}]\in p_{n},t_i\leq t} |\omega_{t_{i+1}}-\omega_{t_{i}}|^{2l}=[\omega]_{t}^{2l}
    \end{equation*}
    for some continuous function \([\omega]^{2l}\).
    For any path \(\omega\in\mathbb{V}_{2l}(\mathfrak{p})\) and \(g\in C^{2l+1}\), we define the rough integral
    \begin{equation*}
        \int_{0}^{T} g'(\omega_{t})\bullet \md \omega_{t}:=\lim_{n\to\infty}\sum_{[t_{i},t_{i+1}]\in p_{n}}\sum_{k=1}^{2l-1}\frac{g^{(k)}(\omega_{t_{i}})}{k!}(\omega_{t_{i+1}}-\omega_{t_{i}})^{k}.
    \end{equation*}
    Then we take \(f(\omega)=\int_{0}^{T}g'(\omega_{t})\bullet\md \omega_{t}\) and the reference measure as the law of the fractional Brownian motion with Hurst parameter \(H=1/2l\).
    By \citet[Theorem 1.5]{cont2019pathwise}, we notice
    \begin{align*}
        \E_{\mu}[f(X)] & = \E_{\mu}\lb[g(X_{T})-g(X_{0})-\frac{1}{(2l)!}\int_{0}^{T}g^{(2l)}(X_{s})\md [X]^{2l}_{s}\rb]           \\
                       & =\E_{\mu}\lb[g(X_{T})-g(X_{0})-\frac{\E_{\mu}[|X_{1}|^{2l}]}{(2l)!}\int_{0}^{T}g^{(2l)}(X_{s})\md s\rb].
    \end{align*}
    For simplicity, we take \(d=1\), \(p=2\), and \(L=+\infty\1_{(1,\infty)}\).
    Therefore, by Theorem \ref{thm-cont}, we derive
    \begin{equation*}
        \Upsilon=\E_{\mu}\lb[\int_{0}^{T} \lb| \E_{\mu}\lb[g'(X_{T}) - \frac{\E_{\mu}[|X_{1}|^{2l}]}{(2l)!}\int_{t}^{T} g^{(2l+1)}(X_{s})\md s \Bigg|\cF_{t} \rb]\rb|^{2}\md t\rb]^{1/2}.
    \end{equation*}
\end{exam}

As mentioned in the introduction, the hyperbolic scaling is not critical for the martingale constraint problem since, by Doob's decomposition theorem, the difference between two martingales is of infinite variation. This gives an infinite transport cost under hyperbolic scaling, i.e., martingale measures are a `totally disconnected' set:
\begin{rmk}
    Let \(p>1\) and \(\mu\in \scrM(\Omega)\) satisfy \(\E_{\mu}[\sup_{t\in[0,T]}|X(t)|^{p}]<\infty\).
    If a martingale measure \(\nu\in\scrM(\Omega)\) satisfies \(d_{\c}(\mu,\nu)<\infty\), then \(\mu=\nu\).
\end{rmk}

As a consequence, we have to zoom out the scaling. Put differently, we need a cost which allows us to alter the quadratic variation of the path.
We focus on the case \(p=2\) and the reference measure \(\mu\) being the Wiener measure, leaving the general case for future studies. Adopting a parabolic scaling, we formally set cost function \(c\) as
\begin{equation*}
    c(\omega,\eta)= \limsup_{N\to\infty}\sum_{n=1}^{N}|\Delta\omega_{nT/N}-\Delta \eta_{nT/N}|^{2}=[\omega-\eta]_{T}.
\end{equation*}
\begin{rmk}
    \label{rmk:typeofpayoffmg}
    Consider formally taking continuous limit of \(\Upsilon_{\M}\) in Theorem \ref{thm-sens-mart} for \(p=2\):
    \begin{align*}
        \lim_{N\to\infty} & L^{*}\lb(\E_{\mu}\lb[\sum_{n=1}^{N}\lb|\E_{\mu}[\MD_{nT/N}f (X)|\cF_{nT/N}]-\E_{\mu}[\MD_{nT/N}f (X)|\cF_{(n-1)T/N}]\rb|_{*}^{2}\rb]^{1/2}\rb)          \\
        =                 & \lim_{N\to\infty} L^{*} \lb(\E_{\mu}\lb[\sum_{n=1}^{N}\int_{(n-1)T/N}^{nT/N} \|\E_{\mu}[\MD_{t}\MD_{nT/N}f (X)|\cF_{t}]\|^{2}_{\bF} \md t\rb]^{1/2}\rb) \\
                          & \approx L^{*} \lb(\E_{\mu}\lb[\int_{0}^{T} \|\E_{\mu}[\MD_{t}\MD_{t+}f (X)|\cF_{t}]\|^{2}_{\bF} \md t\rb]^{1/2}\rb),
    \end{align*}
    where we apply the Clark--Ocone formula in the second line and denote the Frobenius norm by \(\|\cdot\|_{\bF}\).
    Surprisingly, this limit does not always coincide with \(\Upsilon_{\M}\) for some objectives of our interest such as \(f(X)=\frac{1}{2}[X]_{T}\).
    If the above limit were true, we would have \(\Upsilon_{\M}=L^{*}(0)=0\) since \(\MD_{t} f(X)= 0\) for any \(t\in[0,T]\).
    However, a direct computation shows that \(\Upsilon_{\M}=L^{*}(\sqrt{T})\).
\end{rmk}
Motivated by the above remark, we consider objective functionals \(f\) of the form
\begin{equation*}
    f(\omega)=U\lb(\int_{0}^{T}\la h(t,\omega_{\cdot\wedge t}), \md \omega_{t}\ra\rb).
\end{equation*}
We remark that both \(c\) and \(f\) can be defined in a pathwise sense, see \citet{karandikar95pathwise}.

\begin{asmp}
    \label{asmp-wiener}
    We assume that \(h:I\times \cX\to \bbR^{d} \) is bounded and \(U\in C^{2}\) with  a bounded  second derivative.
    There exists \(\gamma \in(1,2)\) such that under any \(\pi\in\Pi_{\bc}(\mu,*)\) with \(\pi(\cX\times\cdot)\in\scrM(\cX)\) the following holds:
    \begin{itemize}
        \item  \(\E_{\pi}[\sup_{t\in[0,T]}|h(t,X_{\cdot\wedge t})-h(t,Y_{\cdot\wedge t})|^{2}]= o(\E_{\pi}[[X-Y]_{T}]^{1/2}).\)

        \item \(
              \E_{\pi}\Biggl[\sup_{t\in[0,T]}|h(t,Y_{\cdot\wedge t})-h(t,X_{\cdot\wedge t})-\int_{0}^{t} \MD_{s}^{\intercal}h(t,X_{\cdot\wedge t})\md^{-}(Y-X)_{s}|^{\gamma}\Biggr]\\ \ = o(\E_{\pi}[[X-Y]_{T}]^{\gamma/2}).
              \)
        \item \(\{\MD h(t,X_{\cdot \wedge t})\}_{t\in[0,T]}\) is \((Y-X)\) uniformly forward--\(\gamma\) integrable.
        \item \(\sup_{s,t\in[0,T]}\|\MD_{s} h(t,X_{\cdot\wedge t})\|_{\bF}\in L^{2\gamma/(2-\gamma)} \).
        \item \(\lim_{t\to s+}\E_{\pi}[\|\MD_{s}h(t,X_{\cdot\wedge t})-\MD_{t}h(t,X_{\cdot\wedge t})\|_{\bF}^{2}]=0\) for any \(s\in[0,T]\).
        \item  For \(H=\int_{0}^{T}\la h(t,X_{\cdot \wedge t}),\md X_{t}\ra\), \(\MD_{s+}H=\lim_{t\to s+} \MD_{t}H\) and \(\MD_{s+}\MD_{s}H= \lim_{t\to s+}\MD_{t}\MD_{s} H\) both converge in \(L^{2}\) for any \(s\in[0,T]\).
    \end{itemize}
\end{asmp}

\begin{thm}
    \label{thm-cont-mart}
    Take \(p=2\) and let \(\mu\) be the Wiener measure.
    Let Assumptions \ref{asmp-pen} and \ref{asmp-wiener} hold, and denote \(H=\int_{0}^{T}\la h(t,X_{\cdot\wedge t}), \md X_{t}\ra \).
    Recall
    \begin{equation}
        V_{\M}(\delta)=\sup_{\nu\in \scrM(\cX)}\lb\{\E_{\nu}\lb[U\lb(H\rb)\rb]-L_{\delta}(d_{\bc}(\mu,\nu)^{1/2})\rb\}.
    \end{equation}
    Then, we have
    \begin{equation*}
        V_{\M}(\delta)=V_{\M}(0)+\Upsilon_{\M}\delta+o(\delta),
    \end{equation*}
    where
    \begin{equation}
        \Upsilon_{\M}=L^{*}\lb(\E_{\mu}\lb[\int_{0}^{T} \| \varphi_{t}\|^{2}_{\bF}\md t\rb]^{1/2}\rb)
    \end{equation}
    and
    \begin{equation}
        \label{eqn-sens-cont-mart}
        \varphi_{s}= \prescript{\mathrm{p}}{}{\lb\{\MD_{s+}\MD_{s}U(H)-  U'(H) \MD_{s}^{\intercal}h(s,X_{\cdot\wedge s}) \rb\}}.
    \end{equation}
\end{thm}

\begin{exam}
    \label{ex:wiener}
    Let  \(\sigma\in C^{1,2}_{b}\). Then \(h(t,\omega)= \sigma(t,\omega_{t})\) satisfies Assumption \ref{asmp-wiener}.   \
    The first condition follows from BDG inequality and the boundedness of \(\partial_{x}\sigma\).
    For the second one, we notice that \(\MD_{s}h(t,X_{\cdot\wedge t})=  \partial_{x}\sigma(t,X_{t})\1_{s\leq t}.\)
    The estimate follows from Taylor expansion and BDG inequality by noticing
    \begin{align*}
         & \lb|h(t,Y_{\cdot\wedge t})-h(t,X_{\cdot\wedge t})-\int_{0}^{t}\MD_{s}^{\intercal}h(t,X_{\cdot\wedge t}) \md^{-}(Y-X)_{s}\rb| \\
         & \qquad = | \sigma(t,Y_{t})- \sigma(t,X_{t}) -\la \partial_{x}\sigma(t,X_{t}), (Y_{t}-X_{t})\ra|
        \leq C |Y_{t}-X_{t}|^{1+\varepsilon},
    \end{align*}
    for some \(\varepsilon\in(0,1)\).
    Taking \(\gamma=2/(1+\varepsilon)\), we obtain the required estimate.
    For the third condition, we notice \(\int_{0}^{T}\MD_{s}h(t,X_{\cdot\wedge t})\md^{-} (Y-X)_{s}= \partial_{x}\sigma (t,X_{t}) (Y_{t}-X_{t})  \) and uniform forward integrablity follows from the boundedness of \(\partial_{x}\sigma\).
    The fourth condition is again satisfied by the boundedness of \(\partial_{x} \sigma\).
    The fifth condition is a consequence of the continuity and the boundedness of \(\partial_{x}\sigma\).
    The last condition follows from the fact that \(\MD_{t}H=\sigma(t,X_{t})+\int_{t}^{T} \partial_{x}\sigma(r,X_{r}) \md X_{r}\) and \(\MD_{t}\MD_{s}H =\partial_{x}\sigma(t,X_{t})+\int_{t}^{T}\partial^{2}_{x}\sigma(r,X_{r})\md X_{r}\) for \(t\geq s\).

\end{exam}

\begin{exam} We continue the theme of Example \ref{ex:Greek} and explore a nonparamteric Greek for option pricing in a continuous time setting. Suppose the price process follows
    \begin{equation*}
        \md S_{t}= S_{t}\sigma(t,X_{t})\md X_{t},
    \end{equation*}
    where \(X\) is a Brownian motion under the reference measure \(\mu\) and \(\sigma\in C^{1,2}_{b}\).
    The payoff of a log contract is  given by \(-\log(S_{T}/S_{0})\) and It\^o's formula gives its price as
    \begin{equation*}
        \E_{\mu}[-\log(S_{T}/S_{0})]=\E_{\mu}\lb[\frac{1}{2}\int_{0}^{T}\sigma(t,X_{t})^{2}\md [X]_{t}\rb]=\E_{\mu}\lb[\frac{1}{2}\lb(\int_{0}^{T}\sigma(t,X_{t})\md X_{t}\rb)^{2}\rb].
    \end{equation*}
    We take \(L=+\infty\1_{(1,\infty)}\) and applying Theorem \ref{thm-cont-mart}, noting its assumptions are satisfied by Example \ref{ex:wiener},
    calculate the nonparametric `Greek' of a log contract as
    \begin{equation*}
        \Upsilon_{\M}=\E_{\mu}\lb[\int_{0}^{T} |\varphi_{t}|^{2} \md t\rb]^{1/2},
    \end{equation*}
    where
    \begin{align*}
        \varphi_{s} & =\prescript{\mathrm{p}}{}{  \lb\{\MD_{s+}\MD_{s}\lb(\frac{1}{2}\lb(\int_{0}^{T}\sigma(t,X_{t})\md X_{t}\rb)^{2}\rb)-\int_{0}^{T}\sigma(t,X_{t})\md X_{t} \partial_{x}\sigma(s, X_{s}) \rb\}} \\
                    & = \sigma(s,X_{s})^{2} + \int_{s}^{T}\E_{\mu}\lb[\lb(\partial_{x}\sigma(t,X_{t})\rb)^{2}+ \sigma(t,X_{t}) \partial_{x}^{2}\sigma(t,X_{t})\Big|\cF_{s}\rb] \md t.
    \end{align*}
    If \(\sigma\) is constant then the price of the log contract is given by \(p(\sigma):=\frac{1}{2}\sigma^2 T\) while the above formula gives \( \Upsilon_{\M}=\sigma^2 \sqrt{T}\). The
    difference in scaling in time comes from two factors: one is similar to what we saw in Example \ref{ex:merton} above, the other is due \(\sigma\) being the annualized volatility, while \(\Upsilon_{\M}\) is not annualized.
    To wit, note that a \(\delta\) perturbation of the underlying Brownian motion, \(X\) to \((1+\delta)X\) corresponds to \(\sqrt{T}\delta\) causal perturbation whilst changing the log contract price from \(p(\sigma)\) to \(p(\sigma(1+\delta))\). Thus, the sensitivity to log contract price corresponding to \(\Upsilon_{\M}\) is
    \[ \lim_{\delta \to 0} \frac{1}{2}\frac{\sigma^2(1+\delta/\sqrt{T})^2 T - \sigma^2T}{\delta} = \sigma^2 \sqrt{T}.\]
    In general, taking time scaling into the account, we would expect \(p'(\sigma)\leq  \frac{\sqrt{T}}{\sigma} \Upsilon_\M \) and the fact that we actually have equality shows that, up to the first order, the worst adversarial change to the dynamics comes from a constant shift to \(\sigma\). Finally, we note that, as in Example \ref{ex:merton}, we could take \(L=+\infty\1_{(\sqrt{T}/\sigma,\infty)}\) leading to \( \Upsilon_{\M}=\sigma T=  p'(\sigma)\).
\end{exam}

\subsection{Extensions and further results}\label{sec:further}

\subsubsection*{Filtrations in discrete time}

In Theorem \ref{thm-sens}, we work with the canonical filtration \((\cF_{t})_{t\in I}\). This was taken for ease of notation and to mimic the continuous time results.
We can extend our result to any enlarged filtration \((\tilde{\cF}_{t})_{t\in I}\), and following the same lines of arguments the sensitivity is again given by  \(\Upsilon=L^{*}\lb(\|\op\mD f\|_{L^{q}(\mu)}\rb)\), where the optional projection is with respect to \((\tilde{\cF}_{t})_{t\in I}\).
The notion of causality introduced in Section \ref{sec-pre} is easily transferred to this setting, see \citet{acciaioCausalOptimalTransport2020} for more details.
In particular, if we equip the reference model with the largest filtration, i.e., \(\tilde{\cF_{t}}=\cF\) for any \(t\in I\), then any coupling is a causal coupling.
Under such an extended setting, we retrieve the static Wasserstein sensitivity \citep{bartlSensitivityAnalysisWasserstein2021} as a specific corollary of Theorem~\ref{thm-sens}.

\subsubsection*{Weak OT objective in discrete time}
Another extension is to consider  an optimal stopping problem
\begin{equation*}
    V_{\os} (0)=\sup_{\tau\in\cT}\E_{\mu}[g(X_{\tau})],
\end{equation*}
and its distributionally robust counterpart
\begin{equation*}
    V_{\os}(\delta)=\sup_{\nu\in \scrP(\cX)}\lb\{\sup_{\tau\in\cT}\E_{\nu}[g(X_{\tau})]-L_{\delta}(d_{\bc}(\mu,\nu)^{1/p})\rb\},
\end{equation*}
where \(\cT\) is the set of \((\cF_{t})\)--stopping time.
The dual representation for \(V_{\os}(\delta)\) is discussed in \cite{jiang24Duality}.
For simplicity, we focus on the two-period case, and by Snell's envelope we notice
\begin{equation*}
    \sup_{\tau\in\cT}\E_{\nu}[g(X_{\tau})]=\E_{\nu}\lb[\max\{g(X_{1}),\E_{\nu}[g(X_{2})|X_{1}]\}\rb]=\E_{\nu}[f(X_{1},\Law(X_{2}|X_{1}))].
\end{equation*}
In this case, the objective \(f:\cX_{1}\times\scrP(\cX_{2})\to\bbR\) is a functional  not only of the state but also of the conditional law of the state. Such problems are referred to as Weak OT and were studied in \cite{gozlan2017kantorovich}. For this two-period setting, it can be shown that
\begin{equation*}
    \Upsilon_{\os}:=\lim_{\delta\to 0}(V_{\os}(\delta)-V_{\os}(0))/\delta=L^{*}\lb(\| (\partial_{x}f,\partial_{\mu}f)\|_{L^{q}(\mu)}\rb ),
\end{equation*}
where \(\partial_{\mu}\) is Lion's derivative, suitably adjusted to the change of coordinates by \(\Delta\). However, to the best of our knowledge, the general \(N\)-period case remains open.

\subsubsection*{Semi-martingale ambiguity in continuous time}
Our results can be extended to a framework where the ambiguity set is the set of semi-martingales around the reference model.
In order to allow both the drift and the volatility ambiguity, we need to decompose the semi-martingale into its finite-variation and its martingale part.
For a process \(X\), we write \(X=X^{a}+X^{m}\)  for its Doob's decomposition.
We adopt the objective in Theorem \ref{thm-sens-mart}:
\begin{equation*}
    f(X)= U(H), \quad \text{where} \quad H=\int_{0}^{T}\la h(t,X_{\cdot\wedge t}), \md X_{t}\ra .
\end{equation*}
Let \(d=1\), \(p=2\), \(\mu\) be the Wiener measure and consider the DRO problem given by
\begin{equation*}
    V_{\mathrm{S.Mart}}(\delta)=\sup_{\nu\in  B_{\delta}} \E_{\nu}[U(H)],
\end{equation*}
where \(B_{\delta}=B^{a}_{\delta}\cap B^{m}_{\delta}\) is the intersection of two bicausal balls given by
\begin{equation*}
    B^{a}_{\delta}=\lb\{\nu\in \scrS(\cX): \inf_{\pi\in\Pi_{\bc}(\mu,\nu)}\E_{\pi}\lb[\|X^{a}-Y^{a}\|_{W_{0}^{1,2}}\rb] \leq \delta\rb\}
\end{equation*}
and
\begin{equation*}
    B^{m}_{\delta}= \lb\{\nu\in\scrS(\cX):\inf_{\pi\in\Pi_{\bc}(\mu,\nu)} E_{\pi}\lb[[X^{m}-Y^{m}]_{T}\rb]^{1/2} \leq \delta\rb\}.
\end{equation*}
Here \(\scrS(\cX)\) denotes the set of semi-martingale measures.
In order to apply Theorem \ref{thm-cont}, a pathwise definition of the It\^o integral is needed.
For instance, we  can adapt the approach in Example \ref{exam-rough}, and consider the objective of the form of
\begin{equation*}
    f(\omega)=U\lb(\int_{0}^{T}g'(\omega_{t})\md \omega_{t}\rb)= U\lb(g(\omega_{T})-g(0) -\frac{1}{2}\int_{0}^{T}g''(\omega_{s})\md[\omega]_{s}\rb).
\end{equation*}
We claim that
\begin{align}
    \label{eqn-fo}
     & \E_{\pi}[f(Y)-f(X)] \nonumber                                                                                                                                                    \\
     & = \E_{\pi}[f(X+ (Y-X)^{a} +(Y-X)^{m})-f(X)] \nonumber                                                                                                                            \\
     & = \E_{\pi}[f(X+(Y-X)^{a})-f(X)] + \E_{\pi}[f(X+(Y-X)^{m})-f(X)]+o(\delta) \nonumber                                                                                              \\
     & = \E_{\pi}\lb[\int_{0}^{T}\lb\la \op\mD f_{t}(X), \md (Y-X)_{t}^{a}  \rb\ra \rb] + \E_{\pi}\lb[\int_{0}^{T}\lb \la \varphi_{t}, \md \lt X,(Y-X)^{m} \rt_{t}\rb\ra\rb]+o(\delta),
\end{align}
where  \(\varphi\) is given in \eqref{eqn-sens-cont-mart}, and where the second equality follows by controlling the higher order terms combining the arguments in the proofs of Theorems \ref{thm-cont} and \ref{thm-cont-mart}, and we omit the details. It follows that
\begin{equation}
    \label{eqn-sens-sm}
    \Upsilon_{\mathrm{S.Mart}}=\Upsilon+\Upsilon_{\M}.
\end{equation}
This result allows us to detail the relation between our results and the recent work of \citet{bartl2023sensitivity}. Therein, the authors consider $L^p$-balls around the drift and the volatility of the reference model under a strong formulation.
We adapted their results to the semi-martingale framework as follows.
Let   \((\Omega,\cF,(\cF_{t}),P)\) be a probability space supporting a Brownian motion \(W\).
Consider a DRO problem given by
\begin{equation*}
    \widetilde{V}_{\mathrm{S.Mart}}(\delta)= \inf_{h\in\scrH}\sup_{(b,\sigma)\in \widetilde{B}_{\delta}}\E_{P}\lb[U\lb(\int_{0}^{T}h_{t}\md X^{b,\sigma}_{t}\rb)\rb],
\end{equation*}
where \(\scrH\) is a set of predictable open-loop controls, \(X_{t}^{b,\sigma}=\int_{0}^{t}b_{s}\md s+ \int_{0}^{t}\sigma_{s}\md W_{s}\), and
\begin{equation*}
    \widetilde{B}_{\delta}=\lb\{(b,\sigma):  \E_{P}\lb[\int_{0}^{T}|b_{t}-\bar b_{t}|^{p}\md t\rb]^{1/p}\leq \delta,  \E_{P}\lb[\lb(\int_{0}^{T}|\sigma_{t}-\bar\sigma_{t}|^{2}\md t \rb)^{p/2}\rb]^{1/p}\leq  \delta\rb\},
\end{equation*}
thus an intersection of $L^p$-balls around the drift and volatility coefficients. \citet{bartl2023sensitivity} show that
\begin{equation}
    \label{eqn-sens-cont-mart-bartl}
    \widetilde{\Upsilon}_{\mathrm{S.Mart}}=\E_{P}\lb[\lb(\int_{0}^{T}|Y_{t}h^{*}_{t}|_{*}^{q}\md t\rb)\rb]^{1/q}+\E_{P}\lb[\lb(\int_{0}^{T}| Z_{t}h^{*}_{t} |^{2}\md t \rb)^{q/2}\rb]^{1/q},
\end{equation}
where \(h^{*}\) is the unique optimal control of the reference model and \((Y,Z)\) is the solution to the BSDE:
\begin{equation*}
    Y_{t} = U'\lb(\int_{0}^{T}h^{*}_{t}\md X_{t}^{\bar b,\bar\sigma}\rb)-
    \int_{t}^{T}Z_{s}\md W_{s},\quad t\in [0,T].
\end{equation*}
Notice that in our framework we do not specify a probability space, and hence we only consider feedback (closed-loop) controls \(h\).
The intersection between the two settings is obtained taking \(\bar b=0\), \(\bar\sigma=\Id\) and \(\scrH=\{h\}\) for some deterministic control \(h\).
While \eqref{eqn-sens-cont-mart-bartl} was obtained only for \(p>3\), we find that \eqref{eqn-sens-cont-mart-bartl} coincides with \eqref{eqn-sens-sm} by plugging \(p=2\).
Roughly speaking, this indicates that the volatility ball \(\widetilde{B}_{\delta}\), while more rigid, is actually equivalent to the bicausal ball \(B_{\delta}\) up to the first order approximation.

\begin{rmk} Above, we took an intersection of a drift ball and a volatility ball to obtain a direct comparison with \citet{bartl2023sensitivity}. However, from the causal-OT point of view, it is more natural to allow a trade-off between the two types of perturbations and to combine them into one cost function. This leads us to consider the bicausal discrepancy given by
    \begin{equation*}
        d_{\bc}(\mu,\nu):= \inf_{\pi\in \Pi_{\bc}(\mu,\nu)}\E_{\pi}\lb[ \|X^{a}-Y^{a}\|_{W^{1,2}_{0}}^{2}+[X^{m}-Y^{m}]_{T}   \rb].
    \end{equation*}
    The corresponding DRO problem is given by
    \begin{equation*}
        V_{\mathrm{S.Mart}}(\delta) = \sup_{\nu\in \scrS(\cX)} \lb\{\E_{\nu}[U(H)]-L_{\delta}(d_{\bc}(\mu,\nu)^{1/2})\rb\}.
    \end{equation*}
    For simplicity, consider \(L= +\infty\1_{(1,\infty)}\) as above.
    Combing the first order approximation \eqref{eqn-fo} with the Cauchy inequality, in this setting we obtain
    \begin{equation*}
        \Upsilon_{\mathrm{S.Mart}}= \sqrt{\Upsilon^{2}+\Upsilon_{\M}^{2}}.
    \end{equation*}
\end{rmk}

\subsection{A general strategy of the proof}
\label{sec-proof}
We present a key elementary lemma which provides a unified framework for the proofs.
An upper bound of the sensitivity is given by estimates \eqref{eqn-lem-growth} and \eqref{eqn-lem-ub}.
A lowerbound of the sensitivity is a consequence of the estimate \eqref{eqn-lem-lb}.
In the following proofs, we will verify  all three estimates respectively in each case.
The growth estimate \eqref{eqn-lem-growth} will follow from the growth assumption of \(f\).
Estimate \eqref{eqn-lem-ub} is an asymptotic estimate when the cost is small and is will be derived from the H\"older inequality or the Kunita--Watanabe inequality.
For estimate \eqref{eqn-lem-lb}, we will construct a sequence of couplings which attain the inequality in \eqref{eqn-lem-ub} asymptotically.

\begin{lem}
    \label{lem-pen}
    Let \(\cP \subseteq \Pi(\mu,*)\) be a given set of couplings and \(c\) a cost such that the following conditions hold: there exists a constant \(C\) such that for any \(\pi\in \cP\)
    \begin{equation}
        \label{eqn-lem-growth}
        \E_{\pi}[f(Y)]\leq C (1 + \E_{\pi}[c(X,Y)]),
    \end{equation}
    there exists \(r\) such that for any \(\pi\in \cP\)
    \begin{equation}
        \label{eqn-lem-ub}
        \E_{\pi}[f(Y)-f(X)]\leq r \E_{\pi}[c(X,Y)]^{1/p} + o(\E_{\pi}[c(X,Y)]^{1/p}),
    \end{equation}
    and for all \(\delta>0\) small enough, there exists \(\pi_{\delta}\in\cP\) such that \(\E_{\pi_{\delta}}[c(X,Y)]\leq u^{p}\delta^{p}\) and
    \begin{equation}
        \label{eqn-lem-lb}
        \E_{\pi_{\delta}}[f(Y)-f(X)]\geq ru \delta + o(\delta),
    \end{equation}
    where \(u\) is given by \(ur-L(u)=L^*(r)\).
    Then under Assumption \ref{asmp-pen}, we have
    \begin{equation*}
        \sup_{\pi\in \cP} \lb\{\E_{\pi}[f(Y)-f(X)]- L_{\delta}(\E_{\pi}[c(X,Y)]^{1/p}) \rb\} = L^{*}(r) \delta +o(\delta).
    \end{equation*}
\end{lem}

\begin{proof}
    Since \(L\) satisfies \(\liminf_{u\to \infty}\frac{L(u)}{u^{p}}=+\infty\), there exists \(M_{1}\) and \(C_1>C\) such that \(L_{\delta}(u)>C_{1}\delta^{1-p} u^{p}\) for any \(\delta <1\) and \(u> M_{1}\).
    Combined with \eqref{eqn-lem-growth}, we see that for any \(\delta<1\) we can restrict to measures with uniformly bounded costs
    \begin{align}
        \sup_{\pi\in \cP} & \lb\{\E_{\pi}[f(Y)-f(X)]- L_{\delta}(\E_{\pi}[c(X,Y)]^{1/p}) \rb\}\label{eq:keylemma_valueeq1}    \\
                          & =  \sup_{\pi\in \cQ} \lb\{\E_{\pi}[f(Y)-f(X)]- L_{\delta}(\E_{\pi}[c(X,Y)]^{1/p}) \rb\},\nonumber
    \end{align}
    where \(\cQ=\cP\cap \{\pi:\E_{\pi}[c(X,Y)]\leq M_{1}\}\).
    By \eqref{eqn-lem-ub}, for \(\pi\in\cQ\) we have
    \begin{equation*}
        \E_{\pi}[f(Y)-f(X)]\leq C_2 \E_{\pi}[c(X,Y)]^{1/p},
    \end{equation*}
    for some constant $C_2$.
    Again by the growth assumption of \(L\), there exists \(M_{2}>1\) such that \(L_{\delta}(u)>C_{2}\delta^{1-p} u^{p}\) for any \(\delta <1\) and \(u> M_{2}\).
    Then on the set of \(\{\pi\in \cQ: \E_{\pi}[c(X,Y)]>M_{2}^{p}\delta^{p}\}\) it holds that
    \begin{align*}
        \E_{\pi} & [f(Y)-f(X)]- L_{\delta}(\E_{\pi}[c(X,Y)]^{1/p})
        \leq C_{2}\E_{\pi}[c(X,Y)]^{1/p}-C_{2}\delta^{1-p}\E_{\pi}[c(X,Y)]                        \\
                 & \leq  C_{2}\E_{\pi}[c(X,Y)]^{1/p}[1- (\E_{\pi}[c(X,Y)]^{1/p}/\delta)^{p-1}]<0.
    \end{align*}
    Therefore, by taking \(\cP_{\delta}=\{\pi\in\cQ: \E_{\pi}[c(X,Y)]\leq M_{2}^{p}\delta^{p}\}\), we obtain the desired estimate from \eqref{eqn-lem-ub}
    \begin{align*}
        \sup_{\pi\in \cP} & \lb\{\E_{\pi}[f(Y)-f(X)]- L_{\delta}(\E_{\pi}[c(X,Y)]^{1/p}) \rb\}                                                         \\
        =                 & \sup_{\pi\in \cP_{\delta}} \lb\{\E_{\pi}[f(Y)-f(X)]- L_{\delta}(\E_{\pi}[c(X,Y)]^{1/p}) \rb\}                              \\
                          & \leq \sup_{\pi\in \cP_{\delta}} \lb\{ r \E_{\pi}[c(X,Y)]^{1/p} - \delta L(\E_{\pi}[c(X,Y)]^{1/p}/\delta) + o (\delta)\rb\}
        \leq L^{*}(r)\delta +o(\delta).
    \end{align*}

    On the other hand, taking \(u\) such that  \(ur-L(u)=L^*(r)\) and \(\pi_{\delta}\) in \eqref{eqn-lem-lb}, we obtain the other inequality, and hence the desired equality:
    \begin{align*}
        \sup_{\pi\in \cP_{\delta}} & \lb\{\E_{\pi}[f(Y)-f(X)]- L_{\delta}(\E_{\pi}[c(X,Y)]^{1/p}) \rb\}
        \geq  r u\delta  - \delta L(u) +o(\delta)=\delta L^*(r)+o(\delta).
    \end{align*}
\end{proof}

\section{Proofs of discrete-time results}
\label{sec-discrete}

\begin{proof}[Proof of Theorem \ref{thm-sens}]
    It is clear by definition that \begin{equation*}
        V(\delta)=\sup_{\pi\in\Pi_{\c}(\mu,*)}\lb\{\E_{\pi}[f(Y)]-L_{\delta}(\E_{\pi}[c_{N}(X,Y)]^{1/p})\rb\}.
    \end{equation*}
    By Lemma \ref{lem-pen} it suffices to verify \eqref{eqn-lem-growth}, \eqref{eqn-lem-ub} and \eqref{eqn-lem-lb} with \(\cP=\Pi_{\c}(\mu,*)\) and
    \begin{equation*}
        r=\E_{\mu}\lb[\sum_{n=1}^{N}\bigl|\E_{\mu}[\mD_{n} f(X)|\cF_{n}]\bigr|_{*}^{q}\rb]^{1/q} =\|\op\mD f\|_{L^{q}(\mu)}.
    \end{equation*}
    We derive \eqref{eqn-lem-growth} by noticing \(f(y)\leq C(1+|x|^{p}+|x-y|^{p})\).
    Let \(\pi\in\Pi_{\c}(\mu,*)\), and we notice by H\"older inequality
    \begin{align*}
        \quad\E_{\pi}[f(Y)-f(X)] & = \int_{0}^{1}\E_{\pi}[\la \mD f(X+\lambda(Y-X)) , \Delta Y-\Delta X\ra]\md \lambda               \\
                                 & =\int_{0}^{1}\E_{\pi}[\la \op\mD f(X+\lambda(Y-X)) , \Delta Y-\Delta X\ra]\md \lambda             \\
                                 & \leq \E_{\pi}[c_{N}(X,Y)]^{1/p}\int_{0}^{1} \|\op\mD f(X+\lambda(Y-X))\|_{L^{q}(\pi)}\md \lambda.
    \end{align*}
    In order to verify \eqref{eqn-lem-ub}, we need to show that for any sequence \(\pi_{n}\) with \(\lim_{n\to\infty}\E_{\pi_{n}}[c_{N}(X,Y)]=0\), it holds that
    \begin{equation}\label{eq:Proof43_secondcond}
        \limsup_{n\to\infty} \int_{0}^{1} \|\op\mD f(X+\lambda(Y-X))\|_{L^{q}(\pi_{n})}\md \lambda\leq \|\op \mD f(X)\|_{L^{q}(\mu)}.
    \end{equation}
    Notice that \(\lim_{n\to\infty}\E_{\pi_{n}}[c_{N}(X,Y)]=0\) implies the convergence of \(\pi_{n}\) to the identical coupling \(\widehat{\pi}=(\Id,\Id)_{\#} \mu\) in \(p\)-Wasserstein distance.
    Therefore, by Assumption \ref{asmp-loss} and Jensen inequality, for any \(\lambda\in [0,1]\), we obtain
    \begin{align*}
        \limsup_{n\to \infty} & \|\op \mD f(X+\lambda(Y-X))-\op\mD f(X)\|_{L^{q}(\pi_{n})}                 \\
        \leq                  & \limsup_{n\to \infty} \| \mD f(X+\lambda(Y-X))-\mD f(X)\|_{L^{q}(\pi_{n})} \\
                              & = \|\mD f(X+\lambda(Y-X  ))-\mD f(X  )\|_{L^{q}(\widehat{\pi})}=0.
    \end{align*}
    Assumption \ref{asmp-loss} and dominated convergence theorem now give the desired inequality:
    \begin{align*}
        \limsup_{n\to\infty}\int_{0}^{1} \|\op\mD f(X+\lambda(Y-X))\|_{L^{q}(\pi_{n})}\md \lambda & \leq \limsup_{n\to\infty}\|\op\mD f(X)\|_{L^{q}(\pi_{n})} \\                      & = \|\op\mD f\|_{L^{q}(\mu)},
    \end{align*}
    where in the last equality, we used the fact that \(\E_{\pi_{n}}[\mD_{n}f(X)|\cF_{n}\otimes \cG_{n}]=\E_{\mu}[\mD_{n}f(X)|\cF_{n}]\) since \(\pi_{n}\in \Pi_{\c}(\mu,*)\).

    It remains to verify \eqref{eqn-lem-lb}. We introduce \(v:\bbR^{d}\to\bbR^{d}\) given by
    \begin{equation}\label{eq:proofdiscrete_funct_v}
        v(e)= (e_{1}|e_{1}|^{q-2},\dots, e_{d}|e_{d}|^{q-2}).
    \end{equation}
    We fix \(u>0\) and let \(\pi_{\delta}=(\Id,\Id + u \delta \Delta^{-1}\circ \Phi)_{\#}\mu\) where
    \begin{equation}\label{eq:Phindef_discrete}
        \Phi_{n}(X)=v(\E_{\mu}[\mD_{n}f(X)|\cF_{n}]), \text{ for } n=1,\dots, N.
    \end{equation}
    By construction, \(\pi_{\delta}\in\Pi_{\c}(\mu,*)\). We compute, using $pq-p=q$,
    \begin{equation*}
        \E_{\pi_{\delta}}[c_{N}(X,Y)]= u^{p}\delta^{p}\E_{\mu}\lb[\sum_{n=1}^{N}|\Phi_{n}(X)|^{p}\rb]=\lb(u \|\op\mD f\|_{L^{q}(\mu)}^{q/p}\rb)^{p}\delta^{p}.
    \end{equation*}
    On the other hand, by the fundamental theorem of calculus and the definition of \(\mD\) in \eqref{eq:def_discrete_MD_1}, we have
    \begin{align*}
        \E_{\pi_{\delta}}[f(Y)-f(X)] & =\E_{\mu}[f(X+ u\delta\Delta^{-1} \circ \Phi(X))-f(X)]                                                      \\
                                     & = u\delta\int_{0}^{1}\E_{\mu}[\la\mD f(X+\lambda u \delta\Delta^{-1}\circ \Phi(X)), \Phi(X)\ra]\md \lambda.
    \end{align*}
    Using the assumed growth estimates we can apply dominated convergence theorem, and taking conditional expectations, we see that the term under the integral, for any fixed \(\lambda\), converges to \(\|\op \mD f \|_{L^{q}(\mu)}^q \) as \(\delta\to 0\), and hence
    \begin{align*}
        \E_{\pi_{\delta}}[f(Y)-f(X)] =  \|\op \mD f \|_{L^{q}(\mu)} \lb(u \|\op\mD f\|_{L^{q}(\mu)}^{q/p}\rb)\delta +o(\delta).
    \end{align*}
    If \(r=\|\op\mD f\|_{L^{q}(\mu)}>0\) then above estimate is equivalent to \eqref{eqn-lem-lb}, where we take \(u\) such that \(L^*(r)=ur-L(u)\).
    If \(\|\op\mD f\|_{L^{q}(\mu)}=0\), then the first part has already implied the sensitivity \(\Upsilon=L^{*}(0)=0\).
\end{proof}

\begin{proof}[Proof of Theorem \ref{thm-sens-mart}]
    Notice that
    \begin{equation*}
        V_{\M}(\delta)=\sup_{\pi\in \Pi_{\bc}(\mu,*), \pi(\cX,\cdot)\in\scrM(\cX)}\lb\{\E_{\pi}[f(Y)]-L_{\delta}(\E_{\pi}[c_{N}(X,Y)^{1/p}])\rb\}.
    \end{equation*}
    By Lemma \ref{lem-pen}, it suffices to verify \eqref{eqn-lem-growth}, \eqref{eqn-lem-ub} and \eqref{eqn-lem-lb} with \(\cP=\{\pi\in \Pi_{\bc}(\mu,*): \pi(\cX\times\cdot)\in\scrM(\cX)\}\) and
    \begin{equation*}
        r=  \inf_{h\in\scrH}\E_{\mu}\lb[\sum_{n=1}^{N}\bigl|\E_{\mu}[\bbD_{n} f|\cF_{n}]-h_{n}\bigr|_{*}^{q}\rb]^{1/q}=\inf_{h\in\scrH}\|\op\mD f-h\|_{L^{q}(\mu)},
    \end{equation*}
    where \(\scrH\) is the set of predictable functionals in \(L^{q}\).
    Estimate \eqref{eqn-lem-growth} follows directly from the unconstrained case.
    Let \(\pi\in\cP\) and \(h\in \scrH\cap C^{2}_{b}\).
    We notice by H\"older inequality
    \begin{align*}
        \quad\E_{\pi}[f(Y)-f(X)] & = \int_{0}^{1}\E_{\pi}[\la \mD f(X+\lambda(Y-X)) , \Delta Y-\Delta X\ra]\md \lambda                     \\
                                 & =\int_{0}^{1}\E_{\pi}[\la \op\mD f(X+\lambda(Y-X)) -h(X), \Delta Y-\Delta X\ra]\md \lambda              \\
                                 & \leq \E_{\pi}[c_{N}(X,Y)]^{1/p}\int_{0}^{1} \|\op\mD f(X+\lambda(Y-X))- h(X)\|_{L^{q}(\pi)}\md \lambda.
    \end{align*}
    Here, the second equality follows from \(\pi\in \Pi_{\bc}(\mu,*)\) and Proposition \ref{prop-mart}. Following the arguments used in the proof of Theorem \ref{thm-sens}, noting \(h\) is bounded, we derive
    \begin{equation*}
        \E_{\pi}[f(Y)-f(X)]\leq \E_{\pi}[c_{N}(X,Y)]^{1/p} \|\op\mD f- h\|_{L^{q}(\mu)} + o(\E_{\pi}[c_{N}(X,Y)]^{1/p}).
    \end{equation*}
    We verify \eqref{eqn-lem-ub} by noticing
    \begin{equation*}
        \inf_{h\in\scrH \cap C^{2}_{b}}\|\op\mD f-h\|_{L^{q}(\mu)}=\inf_{h\in \scrH}\|\op\mD f- h\|_{L^{q}(\mu)}.
    \end{equation*}
    We turn to showing \eqref{eqn-lem-lb}. For \(h^{*}\) we denote the \(L^{q}\) predictable projection of \(\op\mD f\), i.e., the unique optimizer of
    \begin{equation}
        \label{eqn-hstar}
        \inf_{h\in\scrH}J[h]=\inf_{h\in\scrH}\|\op\mD f - h\|_{L^{q}(\mu)}^{q}.
    \end{equation}
    We use $v$ as defined in \eqref{eq:proofdiscrete_funct_v} and define \(\Phi=(0,\Phi_{1},\dots,\Phi_{N}):\cX\to\cX\) where
    \begin{equation*}
        \Phi_{n}(X)=v(\E_{\mu}[\mD_{n}f(X)|\cF_{n}]- h^{*}_{n}) \text{ for } n=1,\dots, N.
    \end{equation*}
    The first variation of \eqref{eqn-hstar} yields for any  \(g\in\scrH\)
    \begin{equation*}
        \delta J[h^{*}](g)=\lim_{\varepsilon\to 0}\frac{1}{\varepsilon}(J[h^{*}+\varepsilon g]-J[h^{*}])=q\E_{\mu}[\la \Phi(X), g\ra]\geq 0,
    \end{equation*}
    which implies that
    \begin{equation}
        \label{eqn-EL}
        \E_{\mu}[\Phi_{n}(X)|\cF_{n-1}]=0.
    \end{equation}
    We fix \(u>0\) and let \(\pi_{\delta}=(\Id,\Id+ u\delta \Delta^{-1}\circ \Phi(X))_{\#}\mu\). A direct computation yields
    \begin{equation*}
        \E_{\pi_{\delta}}[c_{N}(X,Y)]= \lb(u \|\op\mD f -h^{*}\|_{L^{q}(\mu)}^{q/p}\rb)^{p} \delta^{p}.
    \end{equation*}
    On the other hand, we have
    \begin{align*}
        \E_{\pi_{\delta}}[f(Y)-f(X)] & =E_{\mu}[f(X+ u \delta\Delta^{-1} \circ \Phi(X))-f(X)]                                                       \\
                                     & = u\delta \int_{0}^{1}\E_{\mu}[\la\mD f(X+\lambda u \delta\Delta^{-1}\circ \Phi(X)), \Phi(X)\ra]\md \lambda.
    \end{align*}
    Since \(f\) satisfies Assumption \ref{asmp-loss}, the dominated convergence theorem gives
    \begin{align*}
        \lim_{\delta\to 0} \int_{0}^{1}\E_{\mu}[\la\mD f(X+\lambda u \delta\Delta^{-1}\circ \Phi(X)), \Phi(X)\ra]\md \lambda & = \E_{\mu}[\la \mD f (X), \Phi(X)\ra]                                           \\
                                                                                                                             & = \E_{\mu}[\la \op\mD f -h^{*}, \Phi\ra]=\|\op \mD f -h^{*}\|^{q}_{L^{q}(\mu)},
    \end{align*}
    were the second equality is a consequence of \eqref{eqn-EL}.
    This implies that
    \begin{equation}
        \label{eqn-lb-mart}
        \E_{\pi_{\delta}}[f(Y)-f(X)] =  \|\op \mD f -h^{*}\|_{L^{q}(\mu)} \lb(u \|\op\mD f-h^{*}\|_{L^{q}(\mu)}^{q/p}\rb)\delta +o(\delta).
    \end{equation}
    The proof would be complete if we were able to show that the constructed \(\pi_{\delta}\in \cP\).
    By \eqref{eqn-EL}, indeed the second marginal of \(\pi_{\delta}\) is a martingale measure.
    But, in general \(\pi_{\delta}\) may not be a bicausal coupling.
    Instead, by Proposition \ref{prop-dense}, we can approximate \(\pi_{\delta}\) by a bicausal coupling \(\widehat{\pi}_{\delta}\in \cP\). Assumption \ref{asmp-loss} ensures that taking \(\varepsilon\) small enough, e.g., \(\varepsilon=\delta^2\), the estimate \eqref{eqn-lb-mart} still holds if we replace \(\pi_{\delta}\) with \(\widehat{\pi}_{\delta}\).
    This concludes the proof.

\end{proof}
\section{Proofs of continuous-time results}\label{sec-continuous}
\subsection{Hyperbolic scaling}
\label{sec-hyperbolic}

\begin{proof}[Proof of Theorem \ref{thm-cont}]
    Recall that \(\cX=C_{0}([0,T];\bbR^{d})\).
    By Lemma \ref{lem-pen}, it suffices to verify \eqref{eqn-lem-growth}, \eqref{eqn-lem-ub} and \eqref{eqn-lem-lb} with \(\cP= \Pi_{\c}(\mu,*)\) and
    \begin{equation*}
        r=\lb(\int_{0}^{T}\E_{\mu}\lb[|\op\mD_{t}f(X)|_{*}^{q}\rb]\md t\rb)^{1/q}= \|\op \mD f\|_{L^{q}(\mu)}.
    \end{equation*}
    For \eqref{eqn-lem-growth}, we notice
    \begin{equation*}
        |f(y)|\leq C(1+\|y\|_{\infty}^{p})\leq C(1+\|x\|^{p}_{W_{0}^{1,p}}+\|x-y\|^{p}_{W_{0}^{1,p}}).
    \end{equation*}
    Let \(P_{N}=\{0=t_{0}<t_{1}<\cdots<t_{N}=T\}\) be the uniform partition of \([0,T]\).
    We denote the piecewise constant interpolation of a path \(\omega\) by
    \begin{equation*}
        \omega^{N}=\omega_{T}\1_{\{t=T\}}+ \sum_{k=0}^{N-1}\omega_{t_{k+1}}\1_{[t_{k},t_{k+1})}.
    \end{equation*}
    Notice that for any continuous path \(\lim_{n\to\infty} \|\omega-\omega^{N}\|_{\infty}=0\).
    We derive
    \begin{align*}
        \label{eqn-dyadic-ub}
        \E_{\pi}[f(Y)-f(X)] & = \lim_{N\to\infty}\E_{\pi}\lb[f(X+(Y-X)^{N})-f(X)\rb]                                                                                                                                                                 \\
                            & = \lim_{N\to\infty} \E_{\pi}\lb[\int_{0}^{1}\sum_{k=0}^{N-1} \la \mD_{t_{k}}f(X+\lambda(Y-X)^{N}),  \Delta_{t_{k}} (X-Y)\ra\md \lambda \rb]                                                                            \\
                            & = \lim_{N\to\infty} \int_{0}^{1} \E_{\pi}\lb[\int_{0}^{1}\sum_{k=0}^{N-1} \lb\la \E_{\pi}\lb[\mD_{t_{k}}f(X+\lambda(Y-X)^{N})\Big|\cF_{t_{k+1}}\otimes \cG_{t_{k+1}}\rb],  \Delta_{t_{k}} (X-Y)\rb\ra\rb] \md \lambda,
    \end{align*}
    where we used Assumption \ref{asmp-cont}, dominated convergence theorem for the first equality and Proposition \ref{prop-derivative} for the second one.
    Applying H\"older inequality to the right-hand side, and by the reverse Fatou lemma, we obtain
    \begin{align*}
        \E_{\pi}[f(Y)-f(X)] & \leq \E_{\pi}\lb[\limsup_{N\to\infty}\frac{N^{p-1}}{T^{p-1}}\sum_{k=0}^{N-1}|\Delta_{t_{k}} X -\Delta_{t_{k}} Y|^{p}\rb]^{1/p}                                                                           \\
                            & \quad  \times \limsup_{ N\to\infty}\int_{0}^{1}\E_{\pi}\lb[\frac{T}{N}\sum_{k=0}^{N-1}\lb|\E_{\pi}[\mD_{t_{k}}f(X+\lambda(Y-X)^{N})|\cF_{t_{k+1}}\otimes \cG_{t_{k+1}}]\rb|_{*}^{q}\rb]^{1/q}\md \lambda \\
                            & = \E_{\pi}[c(X,Y)]^{1/p} \limsup_{N\to\infty}\int_{0}^{1}\lb(\int_{0}^{T}\varphi^{N}(t,\lambda)\md t\rb)^{1/q}\md \lambda,
    \end{align*}
    where we write
    \begin{equation*}
        \varphi^{N}(t,\lambda):=\sum_{k=0}^{N-1}\E_{\pi}\lb[\lb|\E_{\pi}[\mD_{t_{k}}f(X+\lambda(Y-X)^{N})|\cF_{t_{k+1}}\otimes \cG_{t_{k+1}}]\rb|_{*}^{q}\rb]\1_{[t_{k},t_{k+1})}(t).
    \end{equation*}
    Since \(\mD_{t} f\) has a left limit, we have for any \(\omega,\omega'\in\Omega\) and almost every \(t\)
    \begin{equation*}
        \lim_{N\to\infty}\mD_{\lfloor t\frac{N}{T}\rfloor \frac{T}{N}}f(\omega+\lambda(\omega'-\omega)^{N})=\mD_{t}f(\omega+\lambda(\omega'-\omega)).
    \end{equation*}
    As \(|\mD_{t}f(\omega)|_{*}\leq C (1 +\|\omega\|_{\infty}^{p-1})\), by the dominated convergence theorem, we have the \(L^{q}\)-convergence
    \begin{equation*}
        \lim_{N\to\infty}\E_{\pi}\lb[\lb|\mD_{\lfloor t\frac{N}{T}\rfloor \frac{T}{N}}f(X+\lambda(Y-X)^{N})-\mD_{t}f(X+\lambda(Y-X))\rb|_{*}^{q}\rb]=0.
    \end{equation*}
    Combining the above results with Jensen inequality and the right continuity of the filtration, we have
    \begin{equation}
        \label{eqn-phi}
        \lim_{N\to\infty}\varphi^{N}(t,\lambda)=\E_{\pi}\lb[\lb|\E_{\pi}[\mD_{t}f(X+\lambda(Y-X))|\cF_{t}\otimes \cG_{t}\rb]\rb|_{*}^{q}] \quad \md t\otimes \md \lambda\text{--a.e.}
    \end{equation}
    Therefore, again by the dominated convergence theorem we deduce
    \begin{align*}
        \limsup_{N\to\infty} & \int_{0}^{1}\lb(\int_{0}^{T}\varphi^{N}(t,\lambda)\md t\rb)^{1/q}\md \lambda                                                        \\
        \leq                 & \int_{0}^{1}\lb(\int_{0}^{T}\E_{\pi}[|\E_{\pi}[\mD_{t}f(X+\lambda(Y-X))|\cF_{t}\otimes \cG_{t}]|_{*}^{q}]\md t\rb)^{1/q}\md \lambda \\
                             & = \int_{0}^{1}\| \op\mD f(X+\lambda (Y-X))\|_{L^{q}(\pi)}\md \lambda.
    \end{align*}
    In order to verify \eqref{eqn-lem-ub}, it suffices to show that for any \(\pi_{n}\in\cP\) with \(\lim_{n\to\infty}\E_{\pi_{n}}[c(X,Y)]=0\) it holds
    \begin{equation}
        \label{eqn-ub-cont}
        \limsup_{n\to\infty}\int_{0}^{1}\| \op\mD f(X+\lambda (Y-X))\|_{L^{q}(\pi_{n})}\md \lambda\leq \|\op\mD f\|_{L^{q}(\mu)}.
    \end{equation}
    Since \(\|\cdot\|_{\infty}\) is dominated by \(\|\cdot\|_{W^{1,p}_{0}}\) and \(\mD f\) is continuous with respect to the uniform topology,
    following the same arguments as for \eqref{eq:Proof43_secondcond} in the proof of Theorem \ref{thm-sens} we deduce \eqref{eqn-ub-cont}. It remains to establish \eqref{eqn-lem-lb}.
    We define \(\Phi: \cX\to AC_{0}([0,T];\bbR^{d})\) as
    \begin{equation*}
        \Phi_{t}(X)= \int_{0}^{t}v(\op\mD_{s}f(X))\md s,
    \end{equation*}
    where \(v:\bbR^{d}\to\bbR^{d}\) is, as previously, given by \eqref{eq:proofdiscrete_funct_v}.
    We fix \(u>0\) such that \(L^*(r)=ur-L(u)\).
    For any \(\delta>0\), we construct  \(\pi_{\delta}=(\Id,\Id+ u\delta \Phi)_{\#} \mu\).
    Since \(\Phi\) is adapted, we have \(\pi_{\delta}\in \Pi_{\c}(\mu,*)\).
    A direct computation gives
    \begin{equation*}
        \E_{\pi_{\delta}}[c(X,Y)]= u^{p}\delta^{p}\E_{\mu}\lb[\int_{0}^{T}|v(\op\mD_{t} f(X))|^{p}\md t\rb]= \lb(u \|\op\mD f\|_{L^{q}(\mu)}^{q/p}\rb)^{p}\delta^{p}.
    \end{equation*}
    By Proposition \ref{prop-derivative}, with \(\eta=\Phi(X)\) and noting \(\dot{\eta}_t=v(\op\mD_{t}f(X))\), we derive
    \begin{align*}
        \E_{\pi_{\delta}}[f(Y)-f(X)] & =\E_{\mu}[f(X+u\delta \Phi(X))-f(X)]                                                                                        \\
                                     & = u\delta\int_{0}^{1}\int_{0}^{T}\E_{\mu}[\la \mD_{t}f(X+\lambda u\delta \Phi(X)),v(\op\mD_{t}f(X))\ra ] \md t \md \lambda.
    \end{align*}
    By Assumption \ref{asmp-cont} and dominated convergence theorem, as \(\delta\to 0\), we have
    \begin{align*}
        \lim_{\delta\to 0} \int_{0}^{1}\int_{0}^{T}\E_{\mu}[\la \mD_{t}f(X+\lambda u\delta \Phi(X)),v(\op\mD_{t}f(X))\ra ] \md t\md \lambda =  \int_{0}^{T}\E_{\mu}[|\op\mD_{t}f(X)|^{q}] \md t.
    \end{align*}
    This implies that
    \begin{equation*}
        \E_{\pi_{\delta}}[f(Y)-f(X)]= \|\op \mD f \|_{L^{q}(\mu)} \lb(u \|\op\mD f\|_{L^{q}(\mu)}^{q/p}\rb)\delta +o(\delta),
    \end{equation*}
    and hence \eqref{eqn-lem-lb} holds.
\end{proof}
\begin{rmk}
    It is possible to derive the above result directly from
    \begin{equation*}
        \E_{\pi}[f(Y)-f(X )] = \E_{\pi}\lb[\int_{0}^{1}\int_{0}^{T}\la\mD_{t}f(X+\lambda (Y-X)), (X-Y)'(t)\ra\md t\md \lambda\rb].
    \end{equation*}
    However, we point out that the weak derivative \((X-Y)'\) is only defined up to a Lebesgue null set.
    Taking the above discrete-time approximation approach conveniently allowed us to circumvent the intricate measurability discussion on the optional projection in continuous-time. In addition, our approach has its  independent merits as it builds a clear connection between discrete and continuous time problems.
\end{rmk}

\subsection{Parabolic scaling}
\label{sec-parabolic}

In the proofs that follow, in a chain of inequalities, $C$ may denote a different constant from one line to another. Recall  that
\begin{equation*}
    V_{\M}(\delta)=\sup_{\nu\in \scrM(\cX)}\lb\{\E_{\nu}\lb[U\lb(H\rb)\rb]-L_{\delta}(d_{\bc}(\mu,\nu)^{1/2})\rb\},
\end{equation*}
where \(H=\int_{0}^{T} \la h(t,X_{\cdot\wedge t}),\md X_{t}\ra\).
\begin{lem}
    \label{lem-wiener}
    Let  Assumption \ref{asmp-wiener} hold.
    Then, for any \(\pi\in \Pi_{\bc}(\mu,*)\) with \(\pi(\cX\times \cdot)\in \scrM(\cX)\) we have
    \begin{align*}
         & \E_{\pi}\lb[U\lb(\int_{0}^{T} \la h(t,Y_{\cdot\wedge t}), \md Y_{t}\ra\rb)-U\lb(\int_{0}^{T} \la h(t,X_{\cdot\wedge t}), \md X_{t}\ra \rb)\rb]                                                                 \\
         & =\E_{\pi}\lb[U' \lb(H \rb) \lb(\int_{0}^{T}\la h(t,X_{\cdot\wedge t}), \md (Y-X)_{t}\ra+ \int_{0}^{T}\lb\la \int_{0}^{t}\MD_{s}^{\intercal} h(t,X_{\cdot\wedge t})\md^{-} (Y-X)_{s},\md X_{t}   \rb\ra\rb)\rb] \\
         & \qquad +o\lb(\E_{\pi}\lb[\lb[X-Y\rb]_{T}\rb]^{1/2}\rb).
    \end{align*}
\end{lem}

\begin{proof}
    Since \(\pi\) is bicausal coupling between martingale measures, by Proposition \ref{prop-mart} we have  \((X,Y)\) is a joint martingale under \(\pi\).
    By Assumption \ref{asmp-wiener} and It\^o isometry, we notice
    \begin{align*}
         & \E_{\pi}\lb[\lb| \int_{0}^{T} \la h(t,Y_{\cdot\wedge t}), \md Y_{t}\ra-\int_{0}^{T} \la h(t,X_{\cdot\wedge t}), \md X_{t}\ra \rb|^{2}\rb]                                                                                         \\
         & \quad \leq2\E_{\pi}\Biggl[\biggl|\int_{0}^{T} \la h(t,Y_{\cdot\wedge t}), \md (X-Y)_{t}\ra\biggr|^{2}\Biggr]+ 2\E_{\pi}\Biggl[\biggl|\int_{0}^{T}\la h(t,X_{\cdot\wedge t})-h(t,Y_{\cdot\wedge t}),\md X_{t}\ra\biggr|^{2}\Biggr] \\
         & \qquad \leq C \E_{\pi}\lb[ [X-Y]_{T} \sup_{t\in[0,T]}|h(t,Y_{\cdot\wedge t})|^{2}\rb] + C\E_{\pi}\lb[[X]_{T} \sup_{t\in[0,T]}|h(t,X_{\cdot \wedge t})-h(t,Y_{\cdot\wedge t})|^{2}\rb]                                             \\
         & \qquad = o(\E_{\pi}[[X-Y]_{T}]^{1/2}),
    \end{align*}
    where in the last step, we use the fact that \(h\) is bounded and \([X]_{T}=T\).
    Since \(U\) has a bounded second derivative, we have
    \begin{equation*}
        |U(y)-U(x)- U'  (x)(y-x)|\leq C |y-x|^{2}.
    \end{equation*}
    Together with the previous estimate, this implies that
    \begin{align*}
         & \E_{\pi}\lb[U\lb(\int_{0}^{T} \la h(t,Y_{\cdot\wedge t}), \md Y_{t}\ra\rb)-U\lb(\int_{0}^{T} \la h(t,X_{\cdot\wedge t}), \md X_{t}\ra \rb)\rb]                                       \\
         & =\E_{\pi}\lb[ U' \lb(H \rb)\lb( \int_{0}^{T} \la h(t,Y_{\cdot\wedge t}), \md Y_{t}\ra-\int_{0}^{T} \la h(t,X_{\cdot\wedge t}), \md X_{t}\ra \rb)\rb]  +o(\E_{\pi}[[X-Y]_{T}]^{1/2}).
    \end{align*}
    For simplicity, we write
    \begin{equation*}
        I_{1}=\int_{0}^{T}\la h(t,X_{\cdot\wedge t}),\md (Y-X)_{t}\ra  \text{ and } I_{2}=\int_{0}^{T}\lb\la \int_{0}^{t}\MD_{s}^{\intercal} h(t,X_{\cdot\wedge t})\md^{-} (Y-X)_{s},\md X_{t}\rb\ra.
    \end{equation*}
    Notice that \( U'(H )\) has a finite moment of any order.
    To conclude the proof, it then suffices to show that for \(\gamma\) in Assumption \ref{asmp-wiener}
    \begin{align}
        \label{eqn-fo-ito}
        \E_{\pi}\lb[\lb|  \int_{0}^{T} \la h(t,Y_{\cdot\wedge t}), \md Y_{t}\ra-\int_{0}^{T} \la h(t,X_{\cdot\wedge t}), \md X_{t}\ra - I_{1}-I_{2}\rb|^{\gamma}\rb]  =o (\E_{\pi}[[X-Y]_{T}]^{\gamma/2}).
    \end{align}
    Plugging  \(I_{1}\) and \(I_{2}\) into \eqref{eqn-fo-ito}, we obtain
    \begin{align*}
         & \E_{\pi}\lb[\lb|  \int_{0}^{T} \la h(t,Y_{\cdot\wedge t}), \md Y_{t}\ra-\int_{0}^{T} \la h(t,X_{\cdot\wedge t}), \md X_{t}\ra - I_{1} -I_{2}\rb|^{\gamma}\rb]                                        \\
         & \leq C\E_{\pi}\lb[\lb|\int_{0}^{T}\la h(t,Y_{\cdot\wedge t})-h(t,X_{\cdot\wedge t}), \md (Y-X)_{t}\ra \rb|^{\gamma}\rb]
        \\
         & \qquad + C\E_{\pi} \lb[\lb|\int_{0}^{T}\lb\la h(t,Y_{\cdot\wedge t})-h(t,X_{\cdot\wedge t})-\int_{0}^{t} \MD_{s}^{\intercal}h(t,X_{\cdot\wedge t})\md^{-}(Y-X)_{s},\md X_{t}\rb\ra \rb|^{\gamma}\rb] \\
         & :=J_{1}+J_{2}.
    \end{align*}
    It follows from BDG inequality and H\"older inequality that
    \begin{align*}
        J_{1} & \leq C\E_{\pi}\lb[[X-Y]_{T}^{\gamma/2}\sup_{t\in[0,T]}|h(t,Y_{\cdot\wedge t})-h(t,X_{\cdot\wedge t})|^{\gamma}\rb]                                     \\
              & \leq C \E_{\pi}[[X-Y]_{T}]^{\gamma/2}\E_{\pi}\lb[\sup_{t\in[0,T]}|h(t,Y_{\cdot\wedge t})-h(t,X_{\cdot\wedge t})|^{2\gamma/(2-\gamma)}\rb]^{1-\gamma/2} \\
              & \leq C \E_{\pi}[[X-Y]_{T}]^{\gamma/2}\E_{\pi}\lb[\sup_{t\in[0,T]}|h(t,Y_{\cdot\wedge t})-h(t,X_{\cdot\wedge t})|^{2}\rb]^{1-\gamma/2}                  \\
              & = o (\E_{\pi}[[X-Y]_{T}]^{\gamma/2}).
    \end{align*}
    Here, the third line follows from the boundedness of \(h\) and \(\gamma\in(1,2)\).
    For \(J_{2}\), similarly by BDG inequality  and Assumption \ref{asmp-wiener} we deduce
    \begin{align*}
        J_{2} & \leq C \E_{\pi}\lb[\sup_{t\in[0,T]}\lb|h(t,Y_{\cdot\wedge t})-h(t,X_{\cdot\wedge t})-\int_{0}^{t} \MD_{s}^{\intercal}h(t,X_{\cdot\wedge t})\md^{-}(Y-X)_{s}\rb|^{\gamma} [X]_{T}^{\gamma/2}\rb] \\
              & \leq  o( \E_{\pi}[[X-Y]_{T}]^{\gamma/2}).
    \end{align*}

\end{proof}

\begin{proof}[Proof of Theorem \ref{thm-cont-mart}]
    Recall that \(H=\int_{0}^{T}\la h(t,X_{\cdot\wedge t}), \md X_{t}\ra \)  and \[\varphi_{t}= \prescript{\mathrm{p}}{}{\lb\{ \MD_{t+}\MD_{t}U(H)-  U'(H) \MD_{t}^{\intercal}h(t,X_{\cdot\wedge t})\rb\}}.\]
    We remark \(\MD_{+}\MD U(H)\) is well-defined from the regularity condition of \(h\).
    By Lemma \ref{lem-pen}, it suffices to verify \eqref{eqn-lem-growth}, \eqref{eqn-lem-ub}, and \eqref{eqn-lem-lb} with \(\cP=\{\pi\in\Pi_{\bc}(\mu,*):\Pi(\cX\times\cdot)\in~\scrM(\cX)\}\) and
    \begin{equation*}
        r=\E_{\mu}\lb[\int_{0}^{T} \| \varphi_{t}\|^{2}_{\bF}\md t\rb]^{1/2}.
    \end{equation*}
    To show \eqref{eqn-lem-growth}, we notice that \(h\) is bounded and \(U\) has a bounded second derivative, and derive
    \begin{equation*}
        \E_{\pi}\lb[U\lb(\int_{0}^{T}\la h(t,Y_{\cdot\wedge t}), \md Y_{t}\ra\rb)\rb]\leq C \lb(1+ \E_{\pi}\lb[\lb(\int_{0}^{T}\la h(t,Y_{\cdot\wedge t}), \md Y_{t}\ra\rb)^{2}\rb]\rb)\leq C(1+\E_{\pi}[[Y-X]_{T}]).
    \end{equation*}
    Now, we verify \eqref{eqn-lem-ub}.
    Since \(\MD_{s}h(t,X_{\cdot\wedge t})\) satisfies conditions in Theorem \ref{thm-fubini}, by Theorem~\ref{thm-fubini} we calculate
    \begin{align*}
         & \int_{0}^{T} \la h(s,X_{\cdot\wedge s}), \md^{-}(Y-X)_{s} \ra + \int_{0}^{T}\lb\la \int_{0}^{t}\MD^{\intercal}_{s}h(t,X_{\cdot\wedge t})\md^{-}(Y-X)_{s},\md X_{t}\rb\ra \\
         & = \int_{0}^{T} \la h(s,X_{\cdot\wedge s}), \md^{-}(Y-X)_{s} \ra + \int_{0}^{T} \lb\la \int_{s}^{T}\MD_{s} h(t,X_{\cdot\wedge t}) \md X_{t},\md^{-}(Y-X)_{s}   \rb\ra     \\
         & \qquad -\int_{0}^{T}\lb\la \MD_{s}^{\intercal}h(s,X_{\cdot\wedge s}),\md \lt X,Y-X\rt_{s}\rb\ra_{\bF}                                                                    \\
         & = \int_{0}^{T} \la \MD_{s}H, \md^{-}(Y-X)_{s}\ra-\int_{0}^{T}\lb\la \MD_{s}^{\intercal}h(s,X_{\cdot\wedge s}),\md \lt X,Y-X\rt_{s}\rb\ra_{\bF}.
    \end{align*}
    Therefore, Lemma \ref{lem-wiener} yields
    \begin{align}
        \label{eqn-wiener}
         & \E_{\pi}\lb[U\lb(\int_{0}^{T} \la h(t,Y_{\cdot\wedge t}), \md Y_{t}\ra\rb)-U\lb(\int_{0}^{T} \la h(t,X_{\cdot\wedge t}), \md X_{t}\ra \rb)\rb]                                                               \nonumber \\
         & = \E_{\pi}\lb[ U' (H)\int_{0}^{T} \la \MD_{s} H ,\md^{-}(Y-X)_{s}\ra - U'(H)\int_{0}^{T}\lb\la \MD_{s}^{\intercal}h(s,X_{\cdot\wedge s}),\md \lt X,Y-X\rt_{s}\rb\ra_{\bF}\rb]                                          \\
         & \qquad+ o\lb(\E_{\pi}[c(X,Y)]^{1/2}\rb) \nonumber.
    \end{align}
    Since \(U'(H)\) has a finite moment of any order, a simple application of H\"older inequality gives \( \MD U(H) = U' (H)\MD H  \)  is \((Y-X)\)  forward--\(\gamma'\) integrable for any \(\gamma'\in[1,\gamma)\).
    Furthermore, we notice that \( \MD U(H)\) satisfies the assumption in Proposition \ref{prop-fw}, and hence we deduce
    \begin{equation*}
        \E_{\pi}\lb[ U' (H)\int_{0}^{T}\la \MD_{s}H ,\md^{-}(Y-X)_{s}\ra\rb]=\E_{\pi}\lb[\int_{0}^{T}\la \MD_{s+}\MD_{s}U(H), \md \lt X,Y-X\rt_{s}\ra_{\bF}\rb].
    \end{equation*}
    Plugging the above equality into estimate \eqref{eqn-wiener} yields
    \begin{align}
        \label{eqn-wiener-ub}
         & \E_{\pi}\lb[U\lb(\int_{0}^{T} \la h(t,Y_{\cdot\wedge t}), \md Y_{t}\ra\rb)-U\lb(\int_{0}^{T} \la h(t,X_{\cdot\wedge t}), \md X_{t}\ra \rb)\rb]  \nonumber                                                             \nonumber \\
         & = \E_{\pi}\lb[ \int_{0}^{T}  \la \MD_{s+}\MD_{s}U(H)-  U'(H) \MD_{s}^{\intercal}h(s,X_{\cdot\wedge s}),\md \lt X,Y-X\rt_{s}\ra_{\bF}\rb]              + o\lb(\E_{\pi}[c(X,Y)]^{1/2}\rb) \nonumber                               \\
         & =\E_{\pi}\lb[\int_{0}^{T}\la \varphi_{s},\md \lt X,Y-X\rt_{s}\ra_{\bF}\rb]+ o\lb(\E_{\pi}[c(X,Y)]^{1/2}\rb).
    \end{align}
    Hence, by Kunita--Watanabe inequality, we establish \eqref{eqn-lem-ub} holds as follows:
    \begin{align*}
         & \E_{\pi}\lb[U\lb(\int_{0}^{T} \la h(t,Y_{\cdot\wedge t}), \md Y_{t}\ra\rb)-U\lb(\int_{0}^{T} \la h(t,X_{\cdot\wedge t}), \md X_{t}\ra \rb)\rb]                                                                 \\
         & \leq \E_{\pi}\lb[\int_{0}^{T}\la \varphi_{s} \varphi_{s}^{\intercal},\md \lt X\rt_{s}\ra_{\bF}\rb]^{1/2}\E_{\pi}\lb[\int_{0}^{T}\la \Id,\md \lt Y-X\rt_{s}\ra_{\bF}\rb]^{1/2}+ o\lb(\E_{\pi}[c(X,Y)]^{1/2}\rb) \\
         & \leq \E_{\mu}\lb[\int_{0}^{T} \| \varphi_{t}\|^{2}_{\bF}\md t\rb]^{1/2} \E_{\pi}[c(X,Y)]^{1/2} + o\lb(\E_{\pi}[c(X,Y)]^{1/2}\rb).
    \end{align*}

    We turn now to \eqref{eqn-lem-lb} and define
    \(
    \Phi_{\cdot}=\int_{0}^{\cdot}\varphi_{s}\md X_{s}.
    \)
    We fix \(u>0\) such that \(L^*(r)=ur-L(u)\).
    For \(\delta>0\), set  \(\pi_{\delta}= (\Id,\Id+ u\delta \Phi) _{\#}\mu\).
    By direct computation, we have
    \begin{equation*}
        \E_{\pi_{\delta}}[c(X,Y)]=u^{2}\delta^{2}\E_{\mu}[[\Phi]_{T}]= u^{2}\delta^{2} \E_{\mu} \lb[\int_{0}^{T} \|\varphi_{t}\|_{\bF}^{2}\md t \rb].
    \end{equation*}
    On the other hand, we notice
    \begin{align*}
         & \E_{\pi_{\delta}}\lb[  U' (H)\int_{0}^{T} \la \MD_{s} H, \md^{-}(Y-X)_{s}\ra - U'(H)\int_{0}^{T}\lb\la \MD_{s}^{\intercal}h(s,X_{\cdot\wedge s}),\md \lt X,Y-X\rt_{s}\rb\ra_{\bF}\rb] \\
         & =\E_{\pi_{\delta}}\lb[\int_{0}^{T}\lb\la \varphi_{s} ,\md \lt X,Y-X \rt_{s}\rb\ra_{\bF}\rb]                                                                                           \\
         & =u\delta\E_{\mu}\lb[\int_{0}^{T}\la\varphi_{s}, \varphi_{s} \md \lt X\rt_{s} \ra_{\bF}\rb]= u\delta \E_{\mu}\lb[\int_{0}^{T} \|\varphi_{t}\|_{\bF}^{2}\md t\rb].
    \end{align*}
    Hence, it follows from estimate \eqref{eqn-wiener-ub} that
    \begin{equation}
        \label{eqn-cont-mart}
        \E_{\pi_{\delta}}\lb[U\lb(\int_{0}^{T} \la h(t,Y_{\cdot\wedge t}), \md Y_{t}\ra\rb)-U\lb(\int_{0}^{T} \la h(t,X_{\cdot\wedge t}), \md X_{t}\ra \rb)\rb]= u\delta\E_{\mu}\lb[\int_{0}^{T} \|\varphi_{t}\|_{\bF}^{2}\md t\rb] + o(\delta).
    \end{equation}
    The proof would be complete if it was the case that \(\pi_{\delta}\in\cP\).
    However, in general \(\pi_{\delta}\) is not a bicausal coupling.
    To remedy this, we consider the following approximation to \(\pi_{\delta}\).
    Let \(\varphi^{n}\) be a sequence of bounded predictable processes such that
    \begin{equation*}
        \lim_{n\to\infty}\E_{\mu}\lb[ \int_{0}^{T} \|\varphi_{t}-\varphi^{n}_{t}\|_{\bF}^{2}\md t\rb]=0.
    \end{equation*}
    We construct \(\pi_{\delta}^{n}=(\Id,\Id+u\delta\Phi^{n})_{\#}\mu\), where
    \(
    \Phi^{n}_{\cdot}=\int_{0}^{\cdot}\varphi^{n}_{s}\md X_{s}.
    \)
    Following the same argument as above, we have
    \begin{equation*}
        \E_{\pi^{n}_{\delta}}\lb[U\lb(\int_{0}^{T} \la h(t,Y_{\cdot\wedge t}), \md Y_{t}\ra\rb)-U\lb(\int_{0}^{T} \la h(t,X_{\cdot\wedge t}), \md X_{t}\ra \rb)\rb]= u\delta\E_{\mu}\lb[\int_{0}^{T} \la \varphi,\varphi^{n}\ra_{\bF}\md t\rb] + o(\delta).
    \end{equation*}
    We notice that indeed \(X+ u\delta\Phi^{n}\) is a regular martingale.
    Moreover, for \(\delta\) sufficiently small, \(X+u \delta \Phi^{n}\) is a non-degenerate martingale, and hence \(X+u\delta \Phi^{n}\) has the martingale representation property.
    This implies for sufficiently small \(\delta\), \(\pi^{n}_{\delta}\) is a bicausal coupling by Proposition~\ref{prop-mart}.
    Therefore, there exists a sequence of bicausal couplings \(\widehat{\pi}_{\delta}\) such that  estimate~\eqref{eqn-cont-mart} still holds if we replace \(\pi_{\delta}\) by \(\widehat{\pi}_{\delta}\).
    This concludes the proof.

\end{proof}

\section{Auxiliary proofs}
\label{sec-aux}
In the section, we present the remaining proofs, in particular for the results in Section \ref{sec-tools}.

\begin{proof}[Proof of Proposition \ref{prop-dense}]
    The first part follows directly from \cite[Lemma 3.1]{bartlSensitivityMultiperiodOptimization2022}. We prove the second part by induction.
    The base case \(N=1\) is trivial, and we assume the statement holds for \(N-1\).
    Then there exists \((0,\Phi^{\varepsilon}_{1},\dots,\Phi^{\varepsilon}_{N-1})\) with \(|\Phi^{\varepsilon}_{n}-\Phi_{n}|<\varepsilon\) for \(1\leq n \leq N-1\) such that the projection of  \(\Phi^{\varepsilon}_{\#}\mu\) on the first \(N-1\) marginals is a martingale.
    We may further assume \((0,\Phi^{\varepsilon}_{1},\dots,\Phi^{\varepsilon}_{n-1})\) is an injection for any \(1\leq n\leq N-1\).

    Now, it suffices to construct  \(\Phi^{\varepsilon}_{N}\) with \(|\Phi^{\varepsilon}_{N}-\Phi_{N}|_{\infty}<\varepsilon\) such that \(\Phi^{\varepsilon}_{\#}\mu\in \scrM(\cX)\) and \(\Phi^{\varepsilon}:\cX\to\cX\) is injective.
    We write \(\proj^{\varepsilon}(x)= \lfloor \frac{4x}{\varepsilon}\rfloor \frac{\varepsilon}{4}\) and construct
    \begin{equation*}
        \Phi^{\varepsilon}_{N}(0,x_{1},\dots, x_{N})= \proj^{\varepsilon}\circ \Phi_{N}(0,x_{1},\dots,x_{N}) +\varphi^{\varepsilon }(x_{N})  - r(0,x_{1},\dots,x_{N-1}),
    \end{equation*}
    where \(\proj^{\varepsilon}\) is the projection to the \(\varepsilon/4\) grid, \(\varphi^{\varepsilon}:\bbR^{d}\to (0,\varepsilon/4)^{d}\) is a  measurable bijection, and \(r\) is the residual given by
    \begin{equation*}
        r(0,x_{1},\dots,x_{N-1})=\E_{\mu_{N-1}}[\proj^{\varepsilon}\circ \Phi_{N}(0,x_{1},\dots,x_{N-1},X_{N}) + \varphi^{\varepsilon} (X_{N})]-\Phi_{N-1}^{\varepsilon}(0,x_{1},\dots,x_{N-1}).
    \end{equation*}
    Here, \(\mu_{N-1}\) is the disintegration kernel of \(\mu\) given by
    \begin{equation*}
        \mu(\md x_{1},\dots,\md x_{N})=\mu(\md x_{1},\dots,\md x_{N-1})\mu_{N-1}(x_{1},\dots,x_{N-1},\md x_{N}).
    \end{equation*}
    It is clear that \(|\Phi^{\varepsilon}_{N}-\Phi_{N}|_{\infty}< \varepsilon\) and \(\Phi^{\varepsilon}_{\#}\mu\in \scrM(\cX)\).
    We now verify that \(\Phi^{\varepsilon}\) is injective.
    Assume \(\Phi^{\varepsilon}(0,x_{1},\dots, x_{N})= \Phi^{\varepsilon}(0,x'_{1},\dots,x'_{N})\).
    By induction assumption, we derive \(x_{n}=x'_{n}\) for \(1\leq n\leq N-1\).
    Therefore, this implies
    \begin{equation*}
        \proj^{\varepsilon}\circ \Phi_{N}(0,x_{1},\dots,x_{N}) +\varphi^{\varepsilon }(x_{N})= \proj^{\varepsilon}\circ \Phi_{N}(0,x'_{1},\dots,x'_{N}) +\varphi^{\varepsilon }(x'_{N})
    \end{equation*}
    which further implies \(x_{N}=x'_{N}\).
    Therefore, \(\Phi^{\varepsilon}\) is injective and \((\Id,\Phi^{\varepsilon})_{\#}\mu\in \Pi_{\bc}(\mu,*)\).
\end{proof}

\begin{proof}[Proof of Propositions \ref{prop-derivative} and \ref{prop:pathMDisclassicalMD}]
    We start by proving Proposition \ref{prop-derivative}.
    If \(\eta\) is a simple step function, then the result is immediate from the continuity of \(\mD_{t} f\).
    This implies
    \begin{equation*}
        f(\omega+\eta)-f(\omega)=\int_{0}^{1}\sum_{k=1}^{n}\la \mD_{t_{k}}f(\omega + \lambda \eta ), e_{k}\ra \md \lambda.
    \end{equation*}
    Now, we assume \(\eta\in AC_{0}([0,T];\bbR^{d})\) and take a  sequence of simple step functions \(\eta^{n}(t)=\sum_{k=1}^{n}(t_{k+1}^{n}-t_{k}^{n})e_{k}^{n}\1_{[t_{k}^{n},T]}(t)\) such that
    \begin{equation*}
        \sup_{1\leq k \leq n} (t_{k+1}^{n}-t_{k}^{n})(|e_{k}^{n}|\lor  1) \to 0 \qquad \text{and} \qquad \sum_{k=1}^{n} e_{k}^{n}\1_{[t_{k}^{n},t_{k+1}^{n})}\to \dot{\eta} \text{ in } L^{1}.
    \end{equation*}
    We can always achieve the above by refining  partitions.
    Therefore, we deduce that \(\eta^{n}\) converges to \(\eta\) in the uniform topology.
    This gives
    \begin{align*}
        f(\omega+\eta)-f(\omega) & =\lim_{n\to\infty}f(\omega+\eta^{n})-f(\omega)                                                                                                                            \\
                                 & =\lim_{n\to\infty}\int_{0}^{1}\sum_{k=1}^{n}(t_{k+1}^{n}-t_{k}^{n})\la \mD_{t_{k}^{n}}f(\omega + \lambda \eta^{n} ), e_{k}^{n}\ra \md \lambda                             \\
                                 & = \lim_{n\to\infty}\int_{0}^{1}\int_{0}^{T}\sum_{k=1}^{n}\la \mD_{t_{k}^{n}}f(\omega + \lambda \eta^{n} ),  e_{k}^{n}\1_{[t_{k}^{n},t_{k+1}^{n})}(t)\ra \md t\md \lambda.
    \end{align*}
    Since \(\mD_{t}f\) has left limits, the integrand converges to
    \(\la \mD_{t-}f(\omega + \lambda \eta ), \dot{\eta}_{t}\ra\), which is equal to \(\la \mD_{t}f(\omega + \lambda \eta ), \dot{\eta}_{t}\ra\) \(\md t \otimes \md \lambda\)--a.e.
    Furthermore, as \(\mD f\) is  boundedness preserving, we derive by the dominated convergence theorem that
    \begin{equation*}
        f(\omega+\eta)-f(\omega)=\int_{0}^{1}\int_{0}^{T}\la \mD_{t}f(\omega + \lambda \eta ), \dot{\eta}_{t}\ra \md t\md \lambda,
    \end{equation*}
    from which \eqref{eq:pathMD_char} follows.

    To conclude that the pathwise and the classical Malliavin derivatives coincide, it suffices to observe that \(W^{1,2}_{0}\subseteq~AC_{0}\) and that, by \eqref{eq:classMD_char} and \eqref{eq:pathMD_char}, we have
    \begin{equation*}
        \int_{0}^{T}\la \mD_{t}f(X),\dot{\eta}_{t}\ra \md t = \int_{0}^{T}\la \MD_{t}f(X),\dot{\eta}_{t}\ra \md t,\quad \forall \eta\in W^{1,2}_{0}.
    \end{equation*}
\end{proof}

Before we prove the stochastic Fubini theorem (Theorem \ref{thm-fubini}), we first show that the forward integral indeed agrees with the It\^o integral if the integrand is predictable.
The below can be viewed as an \(L^{\gamma}\) extension of \citet[Proposition 1.1]{russo1993forward}.
\begin{prop}
    \label{prop-forward}
    Let \(\gamma\in[1,2)\), \(A\) be a predictable process, and \(B\) be a regular martingale with \(\md [B]_{t}=\xi_{t}\md t\).
    We assume either
    \begin{equation}
        \E_{P}\lb[\sup_{t\in [0,T]}|A_{t}|^{2\gamma /(2-\gamma)}\rb]<\infty \quad \text{and} \quad \E_{P}\lb[\int_{0}^{T}|\xi_{t}|\md t\rb]<\infty,
    \end{equation}
    or
    \begin{equation}
        \E_{P}\lb[\int_{0}^{T}|A_{t}|^{2}\md t\rb]<\infty \quad \text{and} \quad \E_{P}\lb[ \sup_{t\in[0,T]}|\xi_{t}|^{\gamma/(2-\gamma)}\rb]<\infty.
    \end{equation}
    Then \(A\) is \(B\) forward--\(\gamma\) integrable, and  the forward integral coincides with the It\^o integral, i.e.,
    \begin{equation*}
        \int_{0}^{T}\la A_{t},\md^{-}B_{t}\ra=\int_{0}^{T}\la A_{t},\md B_{t}\ra.
    \end{equation*}

\end{prop}

\begin{proof}
    We write the Hardy--Littlewood maximal process as
    \begin{equation*}
        A^{*}_{t}=\sup_{\varepsilon\in[0,T]} \frac{1}{\varepsilon} \int_{(t-\varepsilon)\vee 0}^{t} |A_{s}|\md s \quad \text{ for } t\in[0,T].
    \end{equation*}
    By Hardy--Littlewood maximal inequality, we have
    \(
    \int_{0}^{T} |A^{*}_{t}|^{2} \md t \leq C \int_{0}^{T} |A_{t}|^{2} \md t,
    \)
    where \(C\) is a deterministic constant.
    Combining this with the assumption and H\"older inequality, we derive either
    \begin{align*}
        \E_{P}\lb[\lb(\int_{0}^{T}|A^{*}_{t}|^{2}\md [B]_{t}\rb)^{\gamma/2}\rb]        \leq  \E_{P}\lb[ \sup_{t\in [0,T]}|A_{t}|^{\gamma} \lb(\int_{0}^{T}|\xi_{t}|\md t\rb)^{\gamma/2} \rb]<\infty,
    \end{align*}
    or
    \begin{equation*}
        \E_{P}\lb[\lb(\int_{0}^{T}|A^{*}_{t}|^{2}\md [B]_{t}\rb)^{\gamma/2}\rb]        \leq  \E_{P}\lb[ \lb(\int_{0}^{T}|A^{*}_{t}|^{2}\md t\rb)^{\gamma/2}\sup_{t\in [0,T]}|\xi_{t}|^{\gamma/2} \rb]<\infty.
    \end{equation*}
    Then by Lebesgue dominated convergence theorem, we obtain
    \begin{align*}
        \lim_{\varepsilon\to 0} \E_{P} \lb[ \lb(\int_{0}^{T}  \lb|\frac{1}{\varepsilon}\int_{(t-\varepsilon)\vee 0}^{t} A_{s}\md s-A_{t} \rb |^{2} \md [B]_{t}\rb)^{\gamma/2}\rb ]=0,
    \end{align*}
    as the integrand converges to 0 from Lebesgue differentiation theorem and is dominated by \(C\lb(\int_{0}^{T}|A^{*}_{t}|^{2}\md [B]_{t}\rb)^{\gamma/2}\).
    Therefore, by BDG inequality, we have the \(L^{\gamma}\) convergence
    \begin{equation*}
        \lim_{\varepsilon\to 0}\int_{0}^{T} \lb\la \frac{1}{\varepsilon}\int_{(t-\varepsilon)\vee 0}^{t}A_{s}\md s, \md B_{t} \rb\ra =\int_{0}^{T}\la A_{t}, \md B_{t}\ra.
    \end{equation*}
    On the other hand, by the Stochastic Fubini theorem \cite[Theorem 2.2]{veraar2012stochastic}, the above limit is equal to
    \begin{align*}
        \lim_{\varepsilon\to 0}\int_{0}^{T} \lb\la \frac{1}{\varepsilon}\int_{(t-\varepsilon)\vee 0}^{t}A_{s}\md s, \md B_{t} \rb\ra & =         \lim_{\varepsilon\to 0} \frac{1}{\varepsilon}\int_{0}^{T}\lb\la\int_{0}^{T} A_{s}\1_{[s,(s+\varepsilon)\wedge T]}(t)\md s,\md B_{t}\rb \ra \\
                                                                                                                                     & =  \lim_{\varepsilon\to 0} \frac{1}{\varepsilon}\int_{0}^{T}\int_{0}^{T}\la A_{s}\1_{[s,(s+\varepsilon)\wedge T]}(t),\md B_{t}\ra \md s              \\
                                                                                                                                     & =\lim_{\varepsilon\to 0} \frac{1}{\varepsilon}\int_{0}^{T}\la A_{t},B_{(t+\varepsilon)\wedge T}-B_{t}\ra\md t                                        \\
                                                                                                                                     & = \int_{0}^{T}\la A_{t},\md^{-}B_{t}\ra.
    \end{align*}
    Therefore, \(A\) is \(B\) forward--\(\gamma\) integrable, and the forward integral coincides with the It\^o integral.

\end{proof}

\begin{proof}[Proof of Theorem \ref{thm-fubini}]
    By the definition of the forward integral, we write
    \begin{align*}
        \quad \int_{0}^{T}\lb\la\int_{s}^{T} \Psi_{s,t}^{\intercal}\md X_{t},\md^{-} M_{s}\rb\ra & = \lim_{\varepsilon\to 0}\frac{1}{\varepsilon} \int_{0}^{T}\lb\la\int_{s}^{T}  \Psi_{s,t}^{\intercal}\md X_{t},(M_{(s+\varepsilon)\wedge T}-M_{s})\rb\ra    \md s                                             \\
                                                                                                 & = \lim_{\varepsilon\to 0} \frac{1}{\varepsilon} \int_{0}^{T}\lb\la\int_{s}^{(s+\varepsilon)\wedge T}  \Psi_{s,t}^{\intercal}\md X_{t},(M_{(s+\varepsilon)\wedge T}-M_{s})\rb\ra    \md s                      \\
                                                                                                 & \qquad +  \frac{1}{\varepsilon} \int_{0}^{T}\lb\la\int_{(s+\varepsilon)\wedge T}^{T}  \Psi_{s,t}^{\intercal}\md X_{t},(M_{(s+\varepsilon)\wedge T}-M_{s})\rb\ra    \md s                                      \\
                                                                                                 & =  \lim_{\varepsilon\to 0}\frac{1}{\varepsilon} \int_{0}^{T}\lb\la\int_{s}^{(s+\varepsilon)\wedge T}  (\Psi_{s,t}^{\intercal}-\Psi_{t,t}^{\intercal})\md X_{t},(M_{(s+\varepsilon)\wedge T}-M_{s})\rb\ra\md s \\
                                                                                                 & \qquad + \lim_{\varepsilon\to 0}\frac{1}{\varepsilon} \int_{0}^{T}\lb\la\int_{s}^{(s+\varepsilon)\wedge T}  \Psi_{t,t}^{\intercal}\md X_{t},(M_{(s+\varepsilon)\wedge T}-M_{s})\rb\ra\md s                    \\
                                                                                                 & \qquad +  \lim_{\varepsilon\to 0}\frac{1}{\varepsilon} \int_{0}^{T}\lb\la\int_{s}^{T}  \Psi_{s,t}(M_{(s+\varepsilon)\wedge t}-M_{s}),\md X_{t}\rb \ra\md s                                                    \\
                                                                                                 & \qquad - \lim_{\varepsilon\to 0} \frac{1}{\varepsilon} \int_{0}^{T}\lb\la\int_{s}^{(s+\varepsilon)\wedge T}  \Psi_{s,t}(M_{ t}-M_{s}),\md X_{t}\rb\ra\md s                                                    \\
                                                                                                 & := J_{1} + J_{2} + J_{3} - J_{4}.
    \end{align*}
    It suffices to show the \(L^{\gamma}\) convergence of \(J_{1}\), \(J_{2}\), \(J_{3}\), and \(J_{4}\) and compute their limits.
    For \(J_{1}\), by  H\"older inequality, BDG inequality, and Fubini theorem, we obtain
    \begin{align*}
         & \lim_{\varepsilon\to 0}\E_{P}\lb[\lb|\frac{1}{\varepsilon}  \int_{0}^{T}\lb\la\int_{s}^{(s+\varepsilon)\wedge T} (\Psi_{s,t}^{\intercal}- \Psi_{t,t}^{\intercal})\md X_{t},M_{(s+\varepsilon)\wedge T}-M_{s}\rb\ra\md s\rb|^{\gamma}\rb]                                                                    \\
         & \leq \lim_{\varepsilon\to 0} \frac{C}{\varepsilon^{\gamma}} \int_{0}^{T} \E_{P}\lb[\lb|\lb\la\int_{s}^{(s+\varepsilon)\wedge T} (\Psi_{s,t}^{\intercal}- \Psi_{t,t}^{\intercal})\md X_{t},M_{(s+\varepsilon)\wedge T}-M_{s}\rb\ra \rb|^{\gamma}\rb]\md s                                                    \\
         & \leq \lim_{\varepsilon\to0} \frac{C}{\varepsilon^{\gamma}} \int_{0}^{T}  \E_{P}\lb[\lb|\int_{s}^{(s+\varepsilon)\wedge T}(\Psi_{s,t}^{\intercal}- \Psi_{t,t}^{\intercal})\md X_{t}  \rb|^{2\gamma/(2-\gamma)}\rb]^{1-\gamma/2} \E_{P}\lb[\lb|M_{(s+\varepsilon)\wedge T}-M_{s}\rb|^{2}\rb]^{\gamma/2} \md s \\
         & \leq \lim_{\varepsilon\to 0}\frac{C}{\varepsilon^{\gamma}} \int_{0}^{T}     \E_{P}\lb[\lb( \int_{s}^{(s+\varepsilon)\wedge T}\| \Psi_{s,t}-\Psi_{t,t}\|_{\bF}^{2}\md t\rb)^{\gamma/(2-\gamma)}\rb]^{1-\gamma/2} \E_{P}\lb[[M]_{(s+\varepsilon)\wedge T}- [M]_{s}\rb]^{\gamma/2}\md s                        \\
         & \leq \lim_{\varepsilon\to 0}  C\lb(\int_{0}^{T}\lb(\frac{1}{\varepsilon}\int_{s}^{(s+\varepsilon)\wedge T}\E_{P}\lb[\|\Psi_{s,t}-\Psi_{t,t}\|_{\bF}^{2}\rb]\md t\rb)^{\gamma/(2-\gamma)} \md s\rb)^{1-\gamma/2}                                                                                             \\
         & \hspace{7cm} \times \lb(\frac{1}{\varepsilon}\int_{0}^{T} \E_{P}\lb[[M]_{(s+\varepsilon)\wedge T}- [M]_{s}\rb] \md s\rb)^{\gamma/2}                                                                                                                                                                         \\
         & \leq \lim_{\varepsilon\to 0}  C\lb(\int_{0}^{T}\lb(\frac{1}{\varepsilon}\int_{s}^{(s+\varepsilon)\wedge T}\E_{P}\lb[\|\Psi_{s,t}-\Psi_{t,t}\|_{\bF}^{2}\rb]\md t\rb)^{\gamma/(2-\gamma)} \md s\rb)^{1-\gamma/2} \E_{P}[[M]_{T}]^{\gamma/2}.
    \end{align*}
    Since we assume  \(
    \sup_{s,t\in[0,T]}\|\Psi_{s,t}\|\in L^{2\gamma/(2-\gamma)},
    \)
    and \(\lim_{t\to s+}\E_{P}[\|\Psi_{s,t}-\Psi_{t,t}\|_{\bF}^{2}]=0\), by Lebesgue dominated convergence theorem we obtain the \(L^{\gamma}\) convergence of
    \begin{equation*}
        J_{1}= \lim_{\varepsilon\to 0}\frac{1}{\varepsilon} \int_{0}^{T}\lb\la\int_{s}^{(s+\varepsilon)\wedge T}  (\Psi_{s,t}^{\intercal}-\Psi_{s,t}^{\intercal})\md X_{t},(M_{(s+\varepsilon)\wedge T}-M_{s})\rb\ra\md s =0.
    \end{equation*}
    We write \(Y_{s}=\int_{0}^{s}\Psi^{\intercal}_{t,t}\md X_{t}\) which is a regular martingale.
    Reorganizing the second term \(J_{2}\), we obtain
    \begin{align*}
        J_{2} & = \lim_{\varepsilon\to 0}\frac{1}{\varepsilon} \int_{0}^{T}\lb\la\int_{s}^{(s+\varepsilon)\wedge T}  \Psi_{t,t}^{\intercal}\md X_{t},M_{(s+\varepsilon)\wedge T}-M_{s}\rb\ra\md s                                                                 \\
              & =\lim_{\varepsilon\to0 }\frac{1}{\varepsilon}\int_{0}^{T}\lb\la Y_{(s+\varepsilon)\wedge T}-Y_{s},M_{(s+\varepsilon)\wedge T}-M_{s}\rb\ra\md s                                                                                                    \\
              & =\lim_{\varepsilon\to0} \frac{1}{\varepsilon}\int_{0}^{T}\lb(\la  Y_{(s+\varepsilon)\wedge T},M_{(s+\varepsilon)\wedge T}\ra-\la Y_{s},M_{s} \ra\rb)\md  s                                                                                        \\
              & \qquad - \lim_{\varepsilon\to 0}\frac{1}{\varepsilon}\int_{0}^{T}\lb\la M_{s}, Y_{(s+\varepsilon)\wedge T}-Y_{s}\rb\ra\md s - \lim_{\varepsilon\to 0} \frac{1}{\varepsilon}\int_{0}^{T}\lb\la Y_{s},M_{(s+\varepsilon)\wedge T}-M_{s}\rb\ra \md s \\
              & = \la Y_{T}, M_{T}\ra -\int_{0}^{T}\la M_{s}, \md Y_{s}\ra -\int_{0}^{T}\la Y_{s}, \md M_{s}\ra                                                                                                                                                   \\
              & =  \int_{0}^{T}\la \Psi_{s,s}, \md \lt X,M\rt_{s} \ra _{\bF}.
    \end{align*}
    The \(L^{\gamma}\) convergence of the second last line is justified from Proposition \ref{prop-forward} by taking \((A,B)=(Y,M)\) and \((A,B)=(M,Y)\) respectively.
    For \(J_{3}\), we interchange the Lebesgue integral and the It\^o integral by \cite[Theorem 2.2]{veraar2012stochastic} and derive
    \begin{align*}
        J_{3} & =  \lim_{\varepsilon\to 0}\frac{1}{\varepsilon} \int_{0}^{T}\lb\la\int_{0}^{t}  \Psi_{s,t}(M_{(s+\varepsilon)\wedge t}-M_{s})\md s,\md X_{t}\rb\ra \\
              & =\int_{0}^{T}\lb\la\int_{0}^{t} \Psi_{s,t}\md^{-} {M}_{s}, \md X_{t}\rb\ra.
    \end{align*}
    The last equality does converge in \(L^{\gamma}\) by combining BDG inequality and the  assumption that \(\{\Psi_{\cdot,t}\}_{t\in[0,T]}\) is \(M_{\cdot}\) uniformly forward--\(\gamma\) integrable.
    For the last term \(J_{4}\), we apply \cite[Theorem 2.2]{veraar2012stochastic} and BDG inequality.
    We derive
    \begin{align*}
         & \lim_{\varepsilon\to 0} \E_{P}\lb[\lb|\frac{1}{\varepsilon} \lb\la \int_{0}^{T}\int_{s}^{(s+\varepsilon)\wedge T}  \Psi_{s,t}(M_{ t}-M_{s}),\md X_{t}\rb \ra\md s\rb|^{\gamma}\rb]                                               \\
         & =\lim_{\varepsilon\to 0} \E_{P}\lb[\lb|\frac{1}{\varepsilon} \int_{0}^{T}\lb\la\int_{(t-\varepsilon)\vee 0}^{t}  \Psi_{s,t}(M_{ t}-M_{s})\md s,\md X_{t}\rb\ra\rb|^{\gamma}\rb]                                                  \\
         & \leq \lim_{\varepsilon\to 0} C \E_{P}\lb[\sup_{s,t\in[0,T]}\|\Psi_{s,t}\|_{\bF}^{\gamma}\lb(\int_{0}^{T}\frac{1}{\varepsilon^{2}}\lb(\int_{t-\varepsilon\vee 0}^{t}|M_{t}-M_{s}| \md s \rb)^{2} \md t \rb)^{\gamma/2}\rb]        \\
         & \leq \lim_{\varepsilon\to 0}  C\lb(\int_{0}^{T} \E_{P}\lb[\sup_{s\in[(t-\varepsilon)\vee 0,t]}|M_{t}-M_{s}|^{2} \rb]\md t\rb)^{\gamma/2} \E_{P}\lb[\sup_{s,t\in[0,T]}\|\Psi_{s,t}\|_{\bF}^{2\gamma/(2-\gamma)}\rb]^{1-\gamma/2}  \\
         & \leq \lim_{\varepsilon\to 0}C \lb(\int_{0}^{T} \E_{P}[[M]_{t}-[M]_{(t-\varepsilon)\vee 0}] \md t\rb)^{\gamma/2}                                                                                                              =0.
    \end{align*}
    Summarizing above estimates, we conclude  \(\int_{\cdot}^{T} \Psi_{\cdot,t}^{\intercal}\md X_{t}\) is \(M_{\cdot}\) forward--\(\gamma\) integrable and
    \begin{align*}
        \int_{0}^{T}\lb\la\int_{s}^{T} \Psi_{s,t}^{\intercal}\md X_{t},\md^{-} M_{s}\rb\ra & =J_{1}+J_{2}+J_{3}-J_{4}                                                                                                                       \\
                                                                                           & =   \int_{0}^{T}\lb\la\int_{0}^{t} \Psi_{s,t} \md^{-} M_{s}, \md X_{t}\rb\ra + \int_{0}^{T}\lb\la \Psi_{t,t},\md \lt X, M\rt_{t}\rb\ra_{\bF} .
    \end{align*}
\end{proof}

\section*{Acknowledgement}
Y. Jiang acknowledges the support by the EPSRC Centre for Doctoral Training in Mathematics of Random Systems: Analysis, Modelling, and Simulation (EP/S023925/1).

\bibliography{transfer}
\end{document}